\documentclass[reqno,11pt]{amsproc}
\usepackage{amssymb}
\usepackage{amsmath}
\usepackage{amsthm}
\usepackage{amsfonts}
\usepackage{microtype}
\usepackage[canadian]{babel}
\usepackage{xcolor}
\usepackage{geometry}
\usepackage{tikz-cd}
\usepackage{enumitem}
\usepackage{ifthen}
\usepackage{hyphenat}
\usepackage{mathtools}

\definecolor{myurlcolor}{rgb}{0,0,0.3}
\definecolor{mycitecolor}{rgb}{0,0.3,0}
\definecolor{myrefcolor}{rgb}{0.3,0,0}
\usepackage[pagebackref,draft=false]{hyperref}
\hypersetup{colorlinks,
linkcolor=myrefcolor,
citecolor=mycitecolor,
urlcolor=myurlcolor}

\usepackage[capitalize]{cleveref}
\usepackage{bookmark}

\newcommand{\beq}{\begin{equation}}
\newcommand{\eeq}{\end{equation}}
\newcommand{\N}{\mathbb{N}}

\newcommand{\Nplus}{\mathbb{N}_{> 0}}			
\newcommand{\Z}{\mathbb{Z}}
\newcommand{\Q}{\mathbb{Q}}
\newcommand{\Qplus}{\mathbb{Q}_{> 0}}			
\newcommand{\R}{\mathbb{R}}
\newcommand{\Rplus}{\mathbb{R}_{> 0}}			
\newcommand{\TR}{\mathbb{TR}}
\newcommand{\TZ}{\mathbb{TZ}}
\newcommand{\B}{\mathbb{B}}				

\newcommand{\op}{\mathrm{op}}
\newcommand{\eps}{\varepsilon}

\newcommand{\Sper}[1]{\mathsf{TSper}(#1)}		

\newcommand{\lev}{\mathrm{lev}}			
\DeclareMathOperator{\Frac}{\mathsf{Frac}}	
\DeclareMathOperator{\Newton}{\mathsf{Newton}}	
\DeclareMathOperator{\supp}{\mathsf{supp}}	

\theoremstyle{plain}
\makeatletter
\newtheorem*{rep@theorem}{\rep@title}
\newcommand{\newreptheorem}[2]{%
\newenvironment{rep#1}[1]{%
 \def\rep@title{#2 \ref{##1}}%
 \begin{rep@theorem}}%
 {\end{rep@theorem}}}
\makeatother

\swapnumbers
\newtheorem{dummy}{Dummy}[section]
\newtheorem{thm}[dummy]{Theorem}\Crefname{thm}{Theorem}{Theorems}
\newreptheorem{thm}{Theorem}
\newtheorem{lem}[dummy]{Lemma}\Crefname{lem}{Lemma}{Lemmas}
\newtheorem{prop}[dummy]{Proposition}\Crefname{prop}{Proposition}{Propositions}
\newtheorem{cor}[dummy]{Corollary}\Crefname{cor}{Corollary}{Corollaries}
\Crefname{conj}{Conjecture}{Conjectures}
\Crefname{qstn}{Question}{Questions}
\newtheorem{defn}[dummy]{Definition}\Crefname{defn}{Definition}{Definitions}
\newtheorem{prob}[dummy]{Problem}\Crefname{prob}{Problem}{Problems}
\newtheorem{nota}[dummy]{Notation}\Crefname{nota}{Notation}{Notations}
\theoremstyle{remark}
\newtheorem{ex}[dummy]{Example}\Crefname{ex}{Example}{Examples}
\newtheorem{rem}[dummy]{Remark}\Crefname{rem}{Remark}{Remarks}
\Crefname{note}{Note}{Notes}
\numberwithin{equation}{section}
\Crefname{enumi}{}{}

\allowdisplaybreaks

\usepackage{enumitem}
\setlist[enumerate]{label=(\alph*),itemsep=5pt,topsep=8pt}
\setlist[itemize]{label=$\triangleright$,itemsep=5pt,topsep=6pt}

\Crefformat{enumi}{#2#1#3}

\let\originalleft\left
\let\originalright\right
\renewcommand{\left}{\mathopen{}\mathclose\bgroup\originalleft}
\renewcommand{\right}{\aftergroup\egroup\originalright}

\usepackage[explicit]{titlesec}

\titleformat{\section}[block]{\bfseries\large\filcenter}{\thesection.}{6pt}{#1}
\titlespacing{\section}{0pt}{18pt}{12pt}
\titleformat{\subsection}[runin]{\bfseries}{\noindent}{0.4em}{#1.}

\usepackage{titletoc}
\titlecontents{part}[0em]
	{\vspace{1pc}}
        {\bfseries\normalsize\contentslabel[\thecontentslabel]{2em}}
	{\bfseries\underline}
	{}
\titlecontents{section}[2em]
	{\vspace{0pt}}
        {\normalfont\normalsize\contentslabel[\thecontentslabel]{2em}}
	{}
	{\titlerule*[.75em]{}\contentspage}

\newcommand{\newterm}[1]{\textbf{#1}}

\title[Abstract Vergleichsstellens\"atze I]{Abstract Vergleichsstellens\"atze\\ for preordered semifields and semirings I}

\author{Tobias Fritz}

\address{Department of Mathematics, University of Innsbruck}
\email{tobias.fritz@uibk.ac.at}



\thanks{\textit{Acknowledgements.}
	We thank Erkka Theodor Haapasalo, Xiaosheng Mu, Tim Netzer and P\'eter Vrana for useful feedback, discussion and suggestions. We also thank Luciano Pomatto, Markus Schweighofer and Omer Tamuz for further discussions.}

\begin{document}

\begin{abstract}
	Real algebra is usually thought of as the study of certain kinds of preorders on fields and rings. Among its core themes are the separation theorems known as \emph{Positivstellens\"atze}.
	However, there is a nascent subfield of real algebra which studies preordered semirings and semifields, which is motivated by applications to probability, graph theory and theoretical computer science, among others.
	Here, we contribute to this subfield by developing a number of foundational results for it, with two abstract \emph{Vergleichsstellens\"atze} being our main theorems.

	Our first Vergleichsstellensatz states that every semifield preorder is the intersection of its total extensions. We apply this to derive our second main result, a Vergleichsstellensatz for certain non-Archimedean preordered semirings in which the homomorphisms to the tropical reals play an important role. We show how this result recovers the existing Vergleichsstellensatz of Strassen and (through the latter) the classical Positivstellensatz of Krivine--Kadison--Dubois.
\end{abstract}

\maketitle

\tableofcontents


\section{Introduction}

Traditionally, real algebra is considered to be the study of ordered fields and ordered rings, broadly construed. Among its central results are a number of separation theorems for such structures known as \emph{Positivstellens\"atze} (and some \emph{Nichtnegativstellens\"atze}).
However, recently there has been increasing interest in similar separation theorems for \newterm{preordered semirings}, in which negative inverses need not exist.\footnote{We assume throughout that the multiplication is commutative. Detailed definitions of all relevant concepts will be given in the main text.} This includes the works of Zuiddam~\cite{zuiddam_thesis}, Vrana~\cite{vrana} and ourselves~\cite{our_spss}. All of these extend an earlier separation theorem for preordered semirings proven by Strassen~\cite{strassen}, which he had developed jointly with a surprising application to the complexity of matrix multiplication, a major open problem in theoretical computer science. The mentioned recent works by Zuiddam, Vrana and ourselves have been motivated by similarly surprising applications to graph theory~\cite{zuiddam,vrana_noncommutative,vrana_probabilistic}, quantum information theory~\cite{PVW_dichotomies} and the complexity theory of tensors in general. We refer to~\cite{WZ} for a recent survey.

Motivated by these developments, the goal of the present paper and its sequel~\cite{partII} is to develop a number of foundational results for the theory of preordered semirings. These results have already found further applications in representation theory~\cite{rep_app}, probability~\cite{arw} and statistics~\cite{major}.
Our first foundational result is a separation theorem for \newterm{preordered semifields}:

\begin{repthm}{abstract_semifield_pss}
	Let $F$ be a preordered semifield. Then its preorder $\le$ is the intersection of all total semifield preorders on $F$ which extend it.
\end{repthm}

By analogy with Positivstellens\"atze and Nichtnegativstellens\"atze, we call this kind of theorem a \newterm{Vergleichsstellensatz} after the German ``\emph{vergleichen}'' (to compare), since it is concerned with the order relations between elements of $F$. Unlike in the case of preordered rings or fields, where the term \emph{Positivstellensatz} applies since the preorder is characterized by its positive cone, a preordered semifield or semiring does not permit the preorder to be characterized through a positive cone. Therefore a separation theorem for these kinds of structures naturally take the form of a Vergleichsstellensatz. For example, \emph{Strassen's Vergleichsstellensatz} refers to the separation theorem of Strassen mentioned above~\cite[Corollary~2.6]{strassen}.

\Cref{abstract_semifield_pss} is analogous to the classical theorem stating that the preorder on a preordered field is the intersection of all the total preorders which extend it. 
Despite the simplicity of the theorem statement, and although our proof is conceptually analogous to the classical one, there are technical intricacies that make \Cref{abstract_semifield_pss} an apparently deeper result than the classical one.
We will conduct the proof in \Cref{semifields} based on a peculiar polynomial identity (namely \Cref{curiouslem}).
It is also worth noting that \Cref{abstract_semifield_pss} is by no means a direct consequence of the classical result, as one can see by noting that a typical semifield does not embed into a field (even if addition is cancellative, see \Cref{function_semifield}).

A very basic result of real algebra is that every Archimedean ordered field embeds into $\R$. One of our easier results is a semifield analogue of this statement:

\begin{repthm}{real_or_tropical}
	Let $F$ be a totally preordered semifield. Then $F$ is multiplicatively Archimedean if and only if it order embeds into one of the following preordered semifields:
	\[
		\R_+, \qquad \R_+^\op, \qquad \TR_+, \qquad \TR_+^\op.
	\]
\end{repthm}

Here, $\TR_+$ denotes the semifield of tropical reals, and the exponent $^\op$ stands for reversing the preorder. (Reversing the preorder on any preordered semiring in our sense produces another preordered semiring.) For this result, our proof is not significantly more difficult than the classical one for $\R$.

\medskip

Before we state our second main result---a Vergleichsstellensatz for preordered semirings---it seems pertinent to give some more detailed motivation for the development of the approach to real algebra presented here.
We explain such motivations in the following two subsections.
The first argues for the intrinsic utility of preordered semifields in the context of real algebra. The second one provides evidence for the ubiquity of preordered semirings and for the rich diversity of situations in which Vergleichsstellens\"atze for preordered semirings can be expected to be applied.

\subsection*{The utility of preordered semifields}

We believe that preordered semifields deserve to be considered among the most fundamental objects of real algebra. The reason for this belief is that there are several distinct benefits to working with semifields rather than fields:

\begin{enumerate}
	\item Evaluation maps on a semifield of rational functions are typically still homomorphisms. For example, let $\R_+[X]$ denote the semiring of single-variable polynomials in one variable with nonnegative coefficients, and let $\R_+(X)$ be the resulting semifield of fractions. This is the semifield of rational functions with coefficients in $\R_+$. Then evaluating at any positive real number defines a semiring homomorphism $\R_+(X) \to \R_+$. Of course, the same does not work at the field level: there is no ring homomorphism $\R(X) \to \R$ at all.
	\item \Cref{abstract_semifield_pss} immediately implies that every \emph{strict} semifield, i.e.~every semifield that is not a field, can be totally preordered. This is in stark contrast to the situation for fields, where a total preorder exists if and only if $-1$ is not a sum of squares.
	\item Again in contrast to fields, semifields often have homomorphisms to the tropical reals $\TR_+$. These compactify the space of homomorphisms to $\R_+$ in many non-Archimedean situations. (See \Cref{chaus}.)
	\item Semifields can combine the convenience of invertibility with aspects of the fine structure often present in local rings, such as the existence of nilpotent elements. For example, let $\R_{(+)}[X]/(X^2)$ be the semiring of real linear functions $r + sX$ with strictly positive constant coefficient $r > 0$, modulo discarding quadratic terms in the formation of products. Then this semiring is a semifield with $(r + sX)^{-1} = r^{-2}(r - sX)$. At the same time, this semiring is the ``positive part'' of the ring of dual numbers $\R[X]/(X^2)$, which is a local ring with a nonzero nilpotent ideal.
	\item In every strict semifield, the set of nonzero elements is closed not only under multiplication, but also under addition (\Cref{strict_semifield}). Hence we can form arbitrary expressions involving nonzero elements, and we do not need to check for invertibility when inverting such an expression. This will be convenient in our proof of \Cref{abstract_semifield_pss} in \Cref{semifields}.
	\item Using semiring homomorphisms as the morphisms of semifields, the category of strict semifields has categorical products: if $F_1$ and $F_2$ are semifields, then their categorical product $F_1 \times F_2$ has underlying set given by a mere subset of the cartesian product, namely
		\[
			(F_1^\times \times F_2^\times) \: \cup \: \{(0,0)\},
		\]
	with the componentwise algebraic structure. There also is a terminal object given by the Boolean semifield $\B$. These observations suggest that the category of strict semifields is much better behaved and more interesting than the category of fields.
\end{enumerate}

Of course, there also are additional technical difficulties posed by working with semifields rather than fields, some of which are related to these advantages. For example, we do not expect there to be a theory of algebraic extensions of semifields that could mirror the elegance and power of Galois theory. But for the purposes of real algebra, we believe that the advantages clearly outweigh the disadvantages. We have therefore come to consider semifields to be our preferred structures for thinking about some of the standard problems that appear in real algebra.

\subsection*{The ubiquity of preordered semirings}

Since semirings are not usually considered in the context of real algebra, we give some further independent motivation for them from a more applied perspective.

Ordered algebraic structures occur in manifold ways throughout mathematics. For example, a very common type of question in all of mathematics is of the following form: does a given object $X$ contain an isomorphic copy of another object $Y$? Similarly, does $X$ have a quotient which is isomorphic to $Y$? Just to give one concrete example where this type of problem occurs, consider the classical problem in representation theory of determining the multiplicity with which some irreducible representation $Y$ occurs in a given representation $X$, say for finite-dimensional complex representations of a fixed compact Lie group. So this asks: what is the largest $n \in \N$ for which the representation $\oplus_{i=1}^n Y$ is a subobject of the representation $X$? Along similar lines, we may ask: is there a third representation $Z$ such that $X \otimes Z$ becomes a subrepresentation of $Y \otimes Z$? Or is there $n \in \N$ such that $X^{\otimes n}$ is a subrepresentation of $Y^{\otimes n}$? These are difficult questions in representation theory. They are quite obviously concerned with the semiring of isomorphism classes of representations, using direct sum as addition, tensor product as multiplication, preordered with respect to inclusion of representations. Vergleichsstellens\"atze for preordered semirings naturally apply here, and they can yield insight that goes far beyond what may be accessible by representation-theoretic methods only.

Of course, the idea of the previous paragraph is not at all specific to representation theory, and we expect that it can be applied in many different ways throughout mathematics.
Generally speaking, the set (or proper class) of isomorphism classes of mathematical objects of a given type forms a preordered semiring as soon as there are notions of ``direct sum'' and ``tensor product'', such that the latter distributes over the former (up to isomorphism), and such that the relevant notion of subobject or quotient object interacts well with these direct sums and tensor products. For example, this is the case for the category of finite-dimensional complex representations of a compact Lie group.
Another case of this type is a certain preordered semiring of graphs, to which Strassen's Vergleichsstellensatz has recently been applied~\cite{zuiddam,mine_graphs}. This has led to some progress in understanding the \emph{Shannon capacity of graphs}, a graph invariant notorious for being difficult to compute.

We aim at developing applications such as these in the future. However, the present work focuses on the real algebra of preordered semifields and preordered semirings itself.

\medskip

We now state our second main result, a new Vergleichsstellensatz for preordered semirings.
The interested reader may keep the above preordered semiring of representations in mind, and notice how the inequalities in \ref{intro_asymp} and \ref{intro_cat} are exactly such that they provide answers to the questions raised above in terms of the inequalities~\ref{intro_phi}.

\begin{repthm}{simpler1}
	Let $S$ be a preordered semiring with $1 \ge 0$ and a power universal element $u$, and let $x, y \in S$ nonzero.
	Then the following are equivalent:
	\begin{enumerate}
		\item\label{intro_phi} $\phi(x) \le \phi(y)$ for all monotone homomorphisms $\phi : S \to \R_+$ or $\phi : S \to \TR_+$.
		\item For every $\eps > 0$, we have
			\[
				x^n \le u^{\lfloor \eps n \rfloor} y^n
			\]
			for all $n \gg 1$.
	\end{enumerate}
	Moreover, suppose that $\phi(x) < \phi(y)$ for all such $\phi$. Then also the following hold:
	\begin{enumerate}[resume]
		\item There is $k \in \N$ such that
			\[
				u^k x^n \le u^k y^n \qquad \forall n \gg 1.
			\]
		\item\label{intro_asymp} If $y$ is power universal as well, then
			\[
				x^n \le y^n \qquad \forall n \gg 1.
			\]
		\item\label{intro_cat} There is nonzero $a \in S$ such that
			\[
				a x \le a y.
			\]
			Moreover, there is $k \in \N$ such that $a \coloneqq u^k \sum_{j=0}^n x^j y^{n-j}$ does the job for any $n \gg 1$.
	\end{enumerate}
\end{repthm}

Here, the assumption that a power universal element $u$ exists amounts to a certain type of \newterm{polynomial growth}: every element must be dominated by a polynomial in $u$. This is a substantial generalization of Archimedeanicity in the usual sense of real algebra.
A genuinely non-Archimedean instance of this theorem is its application to the semiring of Laurent polynomials with nonnegative coefficients (\Cref{ex_laurent}).
On the other hand, specializing to the Archimedean case corresponds to taking $u \coloneqq 2$, in which case the assumption is that every nonzero element $x \in S$ is both lower bounded and upper bounded by a constant. In this case, \Cref{simpler1} specializes to Strassen's Vergleichsstellensatz (\Cref{spss}). We also prove that Strassen's Vergleichsstellensatz in turn implies the classical Positivstellensatz of Krivine--Kadison--Dubois (\Cref{kkd}), which is therefore also contained in \Cref{simpler1}.

\subsection*{Summary}

We briefly summarize the content of each section.

\begin{itemize}
	\item \Cref{semistuff} recalls semirings, semifields and semidomains and gives a few examples in the spirit of our results. \Cref{strict_semifield} is among the core basic observations, stating that a semifield is either a field or zerosumfree. \Cref{quasicompl_defn} introduces a new Archimedeanicity property for semirings.
	\item \Cref{preorderstuff} starts gently by recalling basic order-theoretic definitions, mainly to establish our terminology and notation.
		We then introduce preordered semirings and preordered semifields together with some of their relevant properties. The concept of polynomial growth from \Cref{univ_defn} is a central idea here, as are the preordered semifields of fractions from \Cref{fractions_ordered}.
	\item \Cref{malt_total} defines \emph{multiplicatively Archimedean} preordered semifields, and shows that these all embed into the nonnegative reals $\R_+$ or into the tropical reals $\TR_+$ or their opposites (\Cref{real_or_tropical}). We then use this result to show that there are exactly eight Dedekind complete multiplicatively Archimedean preordered semifields (\Cref{dedekind_class}). This is a semifield analogue of the elementary fact that the real numbers are the only Dedekind complete Archimedean ordered field.

	\item The short \Cref{curioussec} states and proves a peculiar polynomial identity that may be of some independent interest.
	\item \Cref{semifields} uses this polynomial identity to prove \Cref{abstract_semifield_pss}, our abstract Vergleichsstellensatz for preordered semifields, stating that the preorder is the intersection of all its total preorder extensions.
	\item \Cref{n1} uses this result to derive a new Vergleichsstellensatz for preordered semirings (\Cref{simpler1}). The conclusion involving a comparison between large powers relies on our study of a version of the real spectrum for preordered semirings, proving that it is compact Hausdorff (\Cref{chaus}).

		We then develop some consequences of this result, including an extension theorem for monotone homomorphisms (\Cref{extensionthm}), as well as rederivations of existing separation theorems. The latter comprises Strassen's Vergleichsstellensatz, our generalization of it from~\cite{our_spss}, and the classical Positivstellensatz of Krivine--Kadison--Dubois.
\end{itemize}

We note that the material in this paper is largely self-contained and does not require any prior knowledge of real algebra or the theory of semirings, although clearly some familiarity with both will facilitate easier reading.

An important caveat with our definition of multiplicative Archimedeanicity is that the term ``Archimedean'' is used in a sense which matches its standard usage in functional analysis and the theory of ordered abelian groups, \emph{not} in a way which matches its standard usage in real algebra.

\subsection*{Our conventions}

The following conventions, mainly concerning our notation, may be relevant to keep in mind while reading.

\begin{itemize}
	\item All of our algebraic structures, like rings and semirings, are assumed commutative, and we omit further explicit mention of this assumption.
	\item In order to minimize clutter, we frequently state the assumptions on the object under investigation at the beginning of a section or subsection. In those cases, we do not repeat those assumptions in the statement of definitions and theorems except in our main results. Thus when these assumptions are unclear to the reader, taking a look at the beginning of the section or subsection should help.
	\item We usually denote a generic semiring by ``$S$'' (or ``$T$'' when we refer to a second one), and a generic semifield by ``$F$''.
\end{itemize}

\section{Semirings, semifields and semidomains}
\label{semistuff}

We begin by recalling some standard material on semirings, referring to~\cite{golan} for a detailed textbook treatment.

\subsection*{Semirings and homomorphisms}

We give a brief recap of the relevant definitions, emphasizing that we assume commutativity of multiplication without further mention.

\begin{defn}
	A \newterm{semiring} $S$ is a set together with binary operations
	\[
		+ , \, \cdot \: : \: S \times S \longrightarrow S,
	\]
	respectively called \emph{addition} and \emph{multiplication}, and elements $0,1 \in S$ such that both $(S,+,0)$ and $(S,\cdot,1)$ are commutative monoids, and such that multiplication distributes over addition.
\end{defn}

Thus a semiring is like a ring, except in that additive inverses need not exist.

\begin{ex}
	$\N$ with its usual algebraic structures is arguably the simplest interesting example of a semiring. It has no proper subsemirings.
\end{ex}

\begin{ex}
	\label{sos}
	Let $R$ be a ring. Then the set of sums of squares in $R$ is a semiring $\Sigma_R$, since it is closed under addition and multiplication and contains the two neutral elements.

	One can consider instead sums of squares which are ``strictly positive'' in the sense of being elements of the set
	\beq
		\label{sigmaR}
		\Sigma_{+,R} \coloneqq \left\{ 1 + \sum_i x_i^2 \:\bigg|\: x_i \in R \right\} \cup \{0\},
	\eeq
	Then $\Sigma_{+,R}$ is again a semiring.
	In some cases, this modified definition may be preferable since the \emph{polynomial growth} condition of \Cref{univ_defn} may hold for $\Sigma_{+,R}$ but not for $\Sigma_R$.

	Similarly if $A$ is an algebra (over $\R$), then it may be of interest to consider the subsemiring
	\beq
		\label{sigma}
		\Sigma_{+,A} \coloneqq \left\{ r + \sum_i x_i^2 \:\bigg|\: r > 0, \: x_i \in A \right\} \cup \{0\}.
	\eeq
	For example if $A = \R[X_1,\ldots,X_d]$, then this is the semiring of sums of squares polynomials plus a strictly positive constant, together with the zero polynomial.
\end{ex}

In those semirings that are of primary interest to us, no nonzero element has an additive inverse:

\begin{defn}[{\cite[p.~4]{golan}}]
	A semiring $S$ is \newterm{zerosumfree} if for all $x, y \in S$,
	\[
		x + y = 0 \quad \Longrightarrow \quad x = 0 \quad \land \quad y = 0.
	\]
\end{defn}

Equivalently, a semiring is zerosumfree if the set of nonzero elements is closed under addition.

There is a category of semirings and semiring homomorphisms, where the latter are defined in the obvious way as follows.

\begin{defn}
	If $S$ and $T$ are semirings, then a \newterm{semiring homomorphism} from $S$ to $T$ is a map $f : S \to T$ which preserves addition and multiplication, i.e.
	\[
		f(x + y) = f(x) + f(y), \qquad f(xy) = f(x) f(y)
	\]
	for all $x,y \in S$, as well as the neutral elements, $f(0) = 0$ and $f(1) = 1$.
\end{defn}

\begin{ex}
	\label{free_semiring}
	Let $\N[\underline{X}] = \N[X_1,\ldots,X_d]$ be the polynomial semiring in $d$ variables with natural number coefficients.
	Then for any semiring $S$, the homomorphisms $\N[\underline{X}] \to S$ are in bijection with the $d$-tuples $x \in S^d$, where the homomorphism associated to $x \in S^d$ maps a polynomial $p \in \N[\underline{X}]$ to its evaluation $p(x) \in S$.
	Hence $\N[\underline{X}]$ is the free semiring on $d$ generators.
\end{ex}

\subsection*{Semifields and semidomains}

As one would expect, there also is a ``semi-'' version of fields, which similarly are ``fields without negatives''.

\begin{defn}
	\label{semifield_defn}
	A semiring $S$ is a \newterm{semifield} if $S^\times = S \setminus \{0\}$, i.e.~every nonzero element is invertible and in addition $1 \neq 0$.
\end{defn}

We typically denote a semifield by $F$.

\begin{ex}
	\label{basic_semifields}
	Clearly every field is a semifield.
	The most basic examples of semifields that are not fields are the following:
	\begin{itemize} 
		\item The \newterm{rational semifield} $\Q_+$ and the \newterm{real semifield} $\R_+$, where the subscript indicates that these only contain the respective nonnegative numbers, and considering both sets equipped with their usual algebraic operations.
		\item The \newterm{Boolean semifield} $\B \coloneqq \{0,1\}$ with $1 + 1 = 1$, where the rest of the structure is determined uniquely by the semiring axioms.
		\item The \newterm{tropical semifield} $\TR_+ \coloneqq (\R_+,\max,\cdot)$, also known as the \newterm{tropical reals}. This semifield plays a central role in our upcoming results. Note that
			\[
				\log \: : \: \R_+ \longrightarrow \R \cup \{-\infty\}
			\]
			establishes an isomorphism between $\TR_+$ and $(\R\cup\{-\infty\},\max,+)$. We will try to leave it open as much as possible whether the tropical reals are considered in the \newterm{multiplicative picture} as $(\R_+,\max,\cdot)$ or in the \newterm{additive picture} as $(\R\cup\{-\infty\},\max,+)$. But sometimes a concrete choice needs to be made, and which choice is more convenient depend on the particular context. We will explain our choice whenever it matters.
		\item In the additive picture, the tropical semifield contains the subsemifield $\TZ_+ \coloneqq (\Z \cup \{-\infty\},\max,+)$.
	\end{itemize}
	There are some canonical homomorphisms between these semifields:
	\[
		\Q_+ \hookrightarrow \R_+ \twoheadrightarrow \B \hookrightarrow \TZ_+ \hookrightarrow \TR_+ \twoheadrightarrow \B,
	\]
	where $\hookrightarrow$ denotes an obvious inclusion, while $\twoheadrightarrow$ denotes the homomorphisms mapping every nonzero element to $1 \in \B$. 
\end{ex}

While every homomorphism of fields is injective, the same does not apply in the semifield case, as witnessed by the homomorphisms $\R_+ \twoheadrightarrow \B$ and $\TR_+ \twoheadrightarrow \B$.

\begin{ex}
	\label{newtonRplus}
	Let $\N[\underline{X}] = \N[X_1,\ldots,X_d]$ be the polynomial semiring. Then the homomorphisms $\N[\underline{X}] \to \TR_+$ correspond to the $d$-tuples $\alpha \in \TR_+^d$ by \Cref{free_semiring}. Using the additive picture of $\TR_+$ and a multi-index $n = (n_1,\ldots,n_d) \in \N^d$, we have
	\[
		\sum_n r_n \underline{X}^n \longmapsto \max_{n \: : \: r_n > 0} \sum_i \alpha_i n_i.
	\]
	This is exactly the maximal value of the linear function $n \mapsto \sum_i \alpha_i n_i$ over the \newterm{Newton polytope} of the given polynomial. Using the fact that the only homomorphism $\R_+ \to \TR_+$ is the one that factors through $\B$, it follows that the same holds for polynomials with coefficients in $\R_+$. Thus for both $\N$ and $\R_+$ coefficients, considering the value of a polynomial $p$ under all homomorphisms to $\TR_+$ detects the Newton polytope of $p$.
\end{ex}

\begin{ex}
	\label{sosvals}
	On a related note, for any semiring $S$ the homomorphisms $S \to \TR_+$ and $S \to \TZ_+$ bear some similarity to valuations on fields. For example if $\Sigma \subseteq \R[\underline{X}]$ is the semiring of sums of squares polynomials, then the map $\Sigma \to \TZ_+$ which assigns to a polynomial the negative of its degree of vanishing at a given point (or variety) is a homomorphism (using $\TZ_+$ in the additive picture).
\end{ex}

The next two examples of semifields are paradigmatic for our upcoming developments, in the sense that they provide good examples that the reader may want to keep in mind throughout the paper.

\begin{ex}
	\label{function_semifield}
	Let $X$ be a nonempty topological space, and $C(X)_{>0}$ the set of continuous functions $X \to \Rplus$. Then
	\[
		F \coloneqq C(X)_{>0} \cup \{0\}
	\]
	is a semifield with respect to pointwise addition and multiplication. If $X$ is not connected, then there is $f \in F$ with $f^2 + 2 = 3f$ with $f \neq 1,2$, namely a nonconstant function taking the values $1$ and $2$.
	This implies that $F$ does not embed into a field.
\end{ex}

Local rings often contain relatively large semifields as subsemirings, as in the following example.

\begin{ex}
	Let
	\[
		\R_{(+)}[[X]] \coloneqq \left\{ \sum_{i \in \N} r_i X^i \in \R[[X]] \:\Bigg|\: r_0 > 0 \right\} \cup \{0\}
	\]
	be the set of of formal power series with strictly positive constant coefficient together with zero. Since this set of power series is closed under addition and multiplication, and since every nonzero power series with strictly positive constant coefficient has an inverse with strictly positive constant coefficient, we are again dealing with a semifield. This works just as well with the positive cone of any ordered field in place of $\R_+$.
\end{ex}

\begin{lem}
	\label{strict_semifield}
	For a semifield $F$, the following are equivalent:
	\begin{enumerate}[label=(\roman*)]
		\item\label{Ftimesadd} $F$ is zerosumfree.
		\item\label{Funitnoinverse} The unit $1 \in F$ has no additive inverse.
		\item\label{Fnotfield} $F$ is not a field.
	\end{enumerate}
\end{lem}

\begin{proof}
	Condition \ref{Ftimesadd} implies that the unit $1 \in F$ has no additive inverse, since otherwise $1 + (-1) = 0$, forcing $1 = 0$, which is excluded in the definition of semifield. Conversely, if $F$ is not zerosumfree, then we have $x, y\in F^\times$ with $x + y = 0$. Multiplying by $x^{-1}$ gives an additive inverse for $1$.

	Condition \ref{Funitnoinverse} trivially implies that $F$ is not a field. Conversely, if $1$ has an additive inverse $-1$, then multiplying any element $x \in F$ by $-1$ produces the additive inverse $-x$, making $F$ into a field.
\end{proof}

\begin{defn}
	If the equivalent conditions of \Cref{strict_semifield} hold, then we say that the semifield $F$ is \newterm{strict}.
\end{defn}

A convenient property of a strict semifield $F$ is that $F^\times$ is closed both under addition and under multiplication.

Since none of our examples of semifields considered so far is a field, they are all strict semifields. In fact, throughout the paper we will almost exclusively consider strict semifields only. 

\begin{defn}[{\normalfont{e.g.}~\cite{HL}}]
	\label{semidomain}
	A \newterm{semidomain} $S$ is a semiring with $1 \neq 0$ and without zero divisors: for all $x,y \in S$,
	\[
		xy = 0 \quad \Longrightarrow \quad x = 0 \; \lor \; y = 0.
	\]
\end{defn}

\begin{ex}
	\begin{enumerate}
		\item $\N$ is a semidomain, as is every polynomial semiring $\N[\underline{X}]$.
		\item For $n \in \N$, consider
			\[
				S \coloneqq \{x \in \Z^n \mid x = 0 \;\, \lor \;\, \forall i : \: x_i > 0 \}.
			\]
			Then $S$ is a semidomain with respect to the componentwise operations.
	\end{enumerate}
\end{ex}

Trivially every semifield is a semidomain, as is every subsemiring of a semifield. In fact every semidomain is a subsemiring of a semifield, due to the following construction.

\begin{rem}
	\label{fractions}
	If $S$ is a semidomain, then $S$ has a \newterm{semifield of fractions} into which it embeds~\cite[Example~11.7]{golan}.
	It is constructed as the set of equivalence classes of fractions $\frac{x}{y}$ for $x,y \in S$ with $y \neq 0$, where two fractions are considered equivalent if they become equal for some common denominator, and addition and multiplication of fractions is defined as usual. This produces a semifield that we denote by $\Frac(S)$.
	It is strict if and only if $S$ is zerosumfree.

	There is a canonical homomorphism
	\[
		S \longrightarrow \Frac(S), \qquad x \longmapsto \frac{x}{1},
	\]
	which has the obvious universal property with respect to semiring homomorphisms from $S$ into semifields: if $\phi : S \to F$ is a semiring homomorphism into a semifield $F$, then $\phi$ extends uniquely to a homomorphism $\Frac(S) \to F$. 
\end{rem}

\begin{ex}
	\label{semifield_product}
	Let $F_1$ and $F_2$ be strict semifields. Then their \newterm{categorical product} $F_1 \times F_2$ is a semifield with underlying set
	\[
		(F_1^\times \times F_2^\times) \: \cup \: \{(0,0)\},
	\]
	and carries the componentwise algebraic operations. It is easy to check that this again defines a semifield, where this set is closed under addition thanks to the strictness assumption.
	Moreover, it is straightforward to see that the canonical projection homomorphisms
	\[
		F_1 \times F_2 \longrightarrow F_1, \qquad
		F_1 \times F_2 \longrightarrow F_2
	\]
	indeed equip $F_1 \times F_2$ with the universal property of a product in the category of strict semifields and semiring homomorphisms.
	(There also is a coproduct of strict semifields given by the tensor product, but we will not go into the details here.)
\end{ex}

\subsection*{Quasi-complements}

While a semiring does typically not have any additive inverses, many semirings nevertheless satisfy a weaker property that we introduce next. For a semiring $S$, we consider every $n \in \N$ as an element of $S$ via the unique homomorphism $\N \to S$ which maps every $n$ to the $n$-fold sum $1 + \cdots + 1$. Note that $\N \to S$ does not need to be injective: we may have $n = m$ in $S$ even when $n \ne m$ in $\N$. For example in $\TR_+$ or $\B$, all nonzero natural numbers coincide.

\begin{defn}
	\label{quasicompl_defn}
	A semiring $S$ has \newterm{quasi-complements} if for every $x \in S$ there are $n \in \N$ and $y \in S$ with $x + y = n$.
\end{defn}

For example, obviously $\N$ itself has quasi-complements. Less trivially, the semirings $\N^n$ and $\R^n$ with the componentwise algebraic operations have quasi-complements for any $n \in \N$. Quasi-complements are closely related to Archimedeanicity in the usual sense of the term in real algebra (see the proof of \Cref{kkd}).

\section{Preordered semirings and semifields}
\label{preorderstuff}

\subsection{Basic order-theoretic concepts}

We now recall some basic order theory, mainly to introduce our notation and terminology.

\begin{defn}
	A \newterm{preorder relation} $\le$ on a set $X$ is a binary relation that is reflexive and transitive: for all $x,y,z \in X$, we have $x \le x$ and
	\[
		x \le y \quad \land \quad y \le z \quad \Longrightarrow \quad x \le z.
	\]
	It is a \newterm{partial order relation} if it is moreover antisymmetric,
	\[
		x \le y \quad \land \quad y \le x \quad \Longrightarrow \quad x = y.	
	\]
\end{defn}

For example, the \newterm{trivial preorder} on a set $X$ has $x \le y$ if and only if $x = y$.
In our context, the difference between preorders and partial orders has very little relevance: assuming the antisymmetry amounts to using the equality symbol to denote the property that two elements are ordered in both directions. Therefore the difference is merely notational. We usually do not impose antisymmetry, but work with preorders instead: these are notationally more convenient since we frequently extend a given preorder to a larger one, which may then no longer be antisymmetric even if the original preorder was. And instead of forcing antisymmetry in such a situation by taking the respective quotient, it is more convenient to retain the original set and merely extend the ordering relation, accepting that it may not be antisymmetric.

\begin{nota}
	\label{le_notation}
	Let $\le$ be a preorder. Then:
	\begin{enumerate}
		\item $x \ge y$ means $y \le x$.
		\item $x < y$ is shorthand for $x \le y \:\land\: y \not\le x$, and similarly for $x > y$.
		\item We write $\approx$ for the largest equivalence relation contained in $\le$, meaning that
			\[
				x \approx y \quad : \Longleftrightarrow \quad x \le y \quad \land \quad y \le x.	
			\]
		\item We write $\sim$ for the equivalence relation generated by $\le$, that is: $x \sim y$ holds if and only if there is a $z_0, \ldots, z_n$ such that
		$z_i \le z_{i+1}$ or $z_i \ge z_{i+1}$ for all $i$, as well as $x = z_0$ and $z_n = y$.
	\end{enumerate}
\end{nota}

Note that the last two pieces of notation are non-standard. We will use them in this paper and its sequel without further mention.
If $X$ is a preordered set, then the quotient $X / \!\approx$ is a partially ordered set with respect to the induced preorder on equivalence classes.

\begin{nota}
	\label{opposite}
	If $X$ is a preordered set, then we write $X^\op$ for the same set with the \newterm{opposite preorder}: $x \le y$ in $X$ if and only if $y \le x$ in $X^\op$.
\end{nota}

\begin{defn}
	\label{preorder_maps}
	If $X$ and $Y$ are preordered sets, then a map $f : X \to Y$ is:
	\begin{enumerate}
		\item \newterm{monotone} if for all $x,y \in X$,
			\[
				x \le y \quad \Longrightarrow \quad f(x) \le f(y).
			\]
		\item an \newterm{order embedding} if for all $x, y \in X$,
			\[
				x \le y \quad \Longleftrightarrow \quad f(x) \le f(y).
			\]
		\item an \newterm{equivalence} if it is monotone and there is monotone $g : Y \to X$ such that for all $x \in X$ and $y \in Y$,
			\[
				g(f(x)) \approx x, \qquad f(g(y)) \approx y.
			\]
	\end{enumerate}
\end{defn}

Note that an order embedding $X \to Y$ does not need to be injective: $f(x_1) = f(x_2)$ only implies $x_1 \approx x_2$, and the preorder on $X$ does not need to be antisymmetric. Our equivalences are a special case of equivalences in category theory, and the following simple observation is likewise a special case of the standard characterization of equivalence of categories~\cite[Theorem~1.5.9]{riehl}.

\begin{lem}
	\label{equivalence_char}
	For preordered sets $X$ and $Y$, the following are equivalent for a monotone map $f : X \to Y$:
	\begin{enumerate}
		\item $f$ is an equivalence.
		\item $f$ is an order embedding, and for every $y \in Y$ there is $x \in X$ with $f(x) \approx y$.
		\item\label{posetiso} The induced map
			\[
				f / \!\approx \; : \: X / \!\approx\: \longrightarrow Y / \!\approx
			\]
			is an isomorphism of partially ordered sets.
	\end{enumerate}
\end{lem}

In part \ref{posetiso}, monotonicity of $f$ is relevant only for showing that $f / \!\approx$ is well-defined. We could just as well relax the assumption and merely require that $f$ preserve $\approx$, and then monotonicity becomes a consequence of $f / \!\approx$ being an isomorphism.

\begin{defn}
	\label{totalfull}
	A \newterm{total preorder} is a preorder $\le$ such that
	\[
		x \le y \quad \lor \quad y \le x
	\]
	holds for all $x$ and $y$.
\end{defn}

\subsection*{Preordered semirings}

We now introduce the protagonists of this work.

\begin{defn}
	\label{preordered_semiring_defn}
	A \newterm{preordered semiring} is a semiring $S$ together with a preorder relation $\le$ such that addition and multiplication are monotone: for all $a, x, y \in S$,
	\[
		x\le y \qquad \Longrightarrow \qquad a + x \le a + y \quad \land \quad a x \le a y.
	\]
\end{defn}

These compatibility conditions are equivalent to saying that the preorder relation $\le \: \subseteq S \times S$ must be a subsemiring.

\begin{rem}
	We may or may not have $1 \ge 0$ in $S$.\footnote{Note that $1 \ge 0$ is assumed in the definition of preordered semiring in the work of Vrana~\cite{vrana}. We have found it more natural not to make this assumption, since in particular \Cref{abstract_semifield_pss} holds just as well without it, and there are interesting applications of our theory in which it does not hold.} If we do, then we necessarily have $x \ge 0$ for all $x\in S$, since $1 \ge 0$ implies $x \cdot 1 \ge x \cdot 0$. But it may also happen that $1 < 0$, as for example in $\N^\op$.
\end{rem}

\begin{rem}
	Following \Cref{opposite}, given a preordered semiring $S$, the \newterm{opposite preordered semiring} $S^\op$ is the same semiring carrying the opposite preorder. Since the conditions of \Cref{preordered_semiring_defn} are invariant under reversing the order, $S^\op$ is a preordered semiring again.
\end{rem}

\begin{rem}
	\label{generate_order}
	Given a semiring $S$, we can consider the set of all possible preorders on $S$ which make $S$ into a preordered semiring. Since the intersection of a family of semiring preorders on $F$ is again a semiring preorder, it follows that the semiring preorders form a complete lattice under inclusion. Moreover, for any two families of elements $(x_i)_{i \in I}$ and $(y_i)_{i\in I}$, there is a smallest semiring preorder on $S$ such that $x_i \le y_i$ for all $i \in I$. We call it the semiring preorder \newterm{generated by} $(x_i \le y_i)_{i \in I}$.
\end{rem}

\begin{defn}
	A preordered semiring $S$ is \newterm{order cancellative} if it satisfies
	\[
		a + x \le a + y \quad \Longrightarrow \quad x \le y
	\]
	for all $a,x,y \in S$.
\end{defn}

Note that the opposite implication always holds by the definition of preordered semiring.

\begin{lem}
	\label{strict_order_preserve}
	Let $S$ be a preordered semiring which is order cancellative. Then
	\[
		x < y \qquad \Longrightarrow \qquad a + x < a + y
	\]
	holds in $S$.
\end{lem}

\begin{proof}
	Clearly $x < y$ implies $a + x \le a + y$ by monotonicity of addition. If $a + x \ge a + y$ was true as well, then the order cancellativity would give us $x \ge y$, contradicting the assumption.
\end{proof}

We now extend the definition of equivalence of preorders (\Cref{preorder_maps}) to semirings.

\begin{lem}
	\label{preordered_semiring_equivalence}
	For preordered semirings $S$ and $T$, the following are equivalent for a monotone homomorphism $f : S \to T$:
	\begin{enumerate}
		\item $f$ is an order embedding, and for every $y \in T$ there is $x \in S$ with $f(x) \approx y$.
		\item The induced homomorphism
			\[
				f / \! \approx \: \: : \: S / \! \approx \: \longrightarrow T / \! \approx
			\]
			is an isomorphism of partially ordered semirings.
	\end{enumerate}
\end{lem}

\begin{proof}
	Straightforward.
\end{proof}

\begin{defn}
	If $f : S \to T$ satisfies the equivalent conditions of \Cref{preordered_semiring_equivalence}, then we say that $F$ is an \newterm{equivalence}.
\end{defn}

The following result is an analogue of Mac Lane's strictification theorem for monoidal categories (but much simpler). As a technical tool it is not very useful, but it provides an important piece of intuition on how to think about preordered semirings up to equivalence. It also shows that the commonly encountered restriction to polynomial rings in real algebra is not a substantial restriction, though in general one needs to consider polynomials in infinitely many variables.

\begin{prop}
	\label{strictify}
	Every preordered semiring $S$ is equivalent to a polynomial preordered semiring, in finitely many variables if $S$ is finitely generated.
\end{prop}

\begin{proof}
	Let $G \subseteq S$ be any generating set, finite if $S$ is finitely generated, and consider the polynomial semiring $\N[(X_a)_{a \in G}]$. The mapping $X_a \mapsto a$ uniquely extends to a homomorphism $\N[(X_a)_{a \in G}] \to S$. Pulling back the preorder from $S$ to $\N[(X_a)_{a \in G}]$ along this homomorphism makes $\N[(X_a)_{a \in G}]$ into a preordered semiring and the homomorphism into an order embedding. It is an order embedding by construction, and hence an equivalence by \Cref{preordered_semiring_equivalence} and its surjectivity.
\end{proof}

\subsection*{Preordered semifields}

Moving on to preorders on semifields, it turns out to be useful to again add a mild nondegeneracy condition to the definition, analogous to $1 \neq 0$ in the definition of semifield itself (\Cref{semifield_defn}).

\begin{defn}
	\label{preordered_semifield}
	A \newterm{preordered semifield} is a strict semifield $F$ equipped with a semiring preorder $\le$ such that $1 \not\approx 0$. 
\end{defn}

\begin{ex}
	All the semifields from \Cref{basic_semifields}, when equipped with their usual order structures, become totally preordered semifields, and all the homomorphisms listed there are monotone, where the injective ones are order embeddings.
\end{ex}

\begin{rem}
	\label{approx_zero}
	If $F$ is a preordered semifield, then $x \approx 0$ in $F$ implies $x = 0$. For if $x \neq 0$, then $x$ would have to be invertible, and then $x \approx 0$ would imply $1 = x x^{-1} \approx 0$, which is assumed not to be the case by the definition of preordered semifield.

	The only semiring preorder on $F$ that is excluded by the requirement $1 \not\approx 0$ is the complete one: $1 \approx 0$ implies $x \approx 0 \approx y$ for all $x,y \in F$.
\end{rem}

\begin{ex}
	\label{preordered_semifield_product}
	Recall the categorical product of strict semifields from \Cref{semifield_product}. Using the componentwise preorder
	\[
		(x_1, x_2) \le (y_1, y_2) \qquad : \Longleftrightarrow \qquad x_1 \le y_1 \enspace \land \enspace x_2 \le y_2	
	\]
	turns this product semifield into a preordered semifield itself. This is the categorical product in the category of preordered semifields and monotone semiring homomorphisms.
\end{ex}

\begin{ex}
	\label{lexico_prod}
	Let $F_1$ be a totally preordered semifield which is order cancellative, such as $\R_+$, and let $F_2$ be any preordered semifield. Then their \newterm{lexicographic product} $F_1 \ltimes F_2$ also has the set
	\[
		(F_1^\times \times F_2^\times) \cup \{(0,0)\}
	\]
	with the componentwise algebraic structure as its underlying semiring, with preorder relation 
	\[
		(x_1, x_2) \le (y_1, y_2) \quad : \Longleftrightarrow \quad x_1 < y_1 \enspace \lor \enspace ( x_1 \approx y_1 \enspace \land \enspace x_2 \le y_2 ).
	\]
	It is straightforward to show that $F_1 \ltimes F_2$ is again a preordered semifield, where the monotonicity of addition makes use of \Cref{strict_order_preserve}. The identity map is a monotone homomorphism from the categorical product to the lexicographic product, $F_1 \times F_2 \to F_1 \ltimes F_2$.
\end{ex}

\begin{defn}
	\label{preord_semidomain}
	A \newterm{preordered semidomain} $S$ is a preordered semiring which is a semidomain (\Cref{semidomain}) and such that for all $x \in S$,
	\[
		x \approx 0 \quad \Longrightarrow \quad x = 0.
	\]
\end{defn}

Every preordered semifield is also a preordered semidomain (\Cref{approx_zero}).
Next, recall the semifield of fractions from \Cref{fractions}.

\begin{lem}
	\label{fractions_ordered}
	If $S$ is a zerosumfree preordered semidomain, then $\Frac(S)$ becomes a preordered semifield with respect to the preorder given by, for nonzero $a,b \in S$,
	\[
		\frac{x}{a} \le \frac{y}{b} \quad :\Longleftrightarrow \quad \exists r \in S \setminus \{0\}: \: xbr \le yar.
	\]
\end{lem}

\begin{proof}
	We already know by \Cref{fractions} that $\Frac(S)$ is a strict semifield, so we only need to show that the above defines a semiring preorder and that $1 \not \approx 0$. We do this in several steps.
	\begin{enumerate}
		\item The preorder relation is well-defined.
			
			Indeed suppose that $\frac{x_1}{a_1}$ and $\frac{x_2}{a_2}$ represent the same element of $\Frac(S)$, meaning that there is nonzero $s$ with $x_1 a_2 s = x_2 a_1 s$. Then if we have $\frac{x_1}{a_1} \le \frac{y}{b}$ by virtue of $x_1 br \le ya_1 r$, then also
			\[
				x_2 b (a_1 s r) = x_1 a_2 b s r \le y a_2 (a_1 s r),
			\]
			which gives $\frac{x_2}{a_2} \le \frac{y}{b}$ since $a_1 s r \neq 0$. Well-definedness with respect to the fraction on the right-hand side works analogously.
		\item The preorder relation is transitive.
	
			Thus suppose that $\frac{x}{a} \le \frac{y}{b} \le \frac{z}{c}$ for nonzero $a,b,c \in S$, meaning that there are nonzero $r,s \in S$ with
			\[
				xbr \le yar, \qquad ycs \le zbs.
			\]
			Then also $brs$ is nonzero, and
			\[
				xcbrs \le yacrs \le zabrs,
			\]
			which gives indeed $\frac{x}{a} \le \frac{z}{c}$.
		\item Multiplication is monotone.
			
			We multiply the assumed inequality $\frac{x}{a} \le \frac{y}{b}$, as witnessed by $xbr \le yar$, by $\frac{z}{c}$. This gives $\frac{xz}{ac} \le \frac{yz}{bc}$ thanks to
			\[
				(xz)(bc)r \le (yz)(ac)r,
			\]
			which is enough.
		\item Addition is monotone.
			
			We use the same assumptions and need to show $\frac{xc + za}{ac} \le \frac{yc + zb}{bc}$, which amounts to
			\[
				(xc + za)(bc)r \le (yc + zb)(ac)r,
			\]
			which indeed follows from the assumption $xbr \le yar$.
		\item We have $\frac{1}{1} \not\approx \frac{0}{1}$ in $\Frac(S)$. Indeed $\frac{1}{1} \ge \frac{0}{1}$ and $\frac{1}{1} \le \frac{0}{1}$ would mean that there are nonzero $r,s \in S$ such that $r \ge 0$ and $s \le 0$. But this gives $0 \le rs \le 0$, and hence $rs = 0$ by \Cref{preord_semidomain}. But then also $r = 0$ or $s = 0$ by the definition of semidomain, contradicting the assumption $r,s \neq 0$. \qedhere
	\end{enumerate}
\end{proof}

\begin{lem}
	\label{frac_order_embed}
	If $S$ is a preordered semidomain, then the canonical homomorphism
	\[
		S \longrightarrow \Frac(S), \qquad x \longmapsto \frac{x}{1}
	\]
	is monotone and satisfies
	\[
		\frac{x}{1} \le \frac{y}{1} \quad \Longleftrightarrow \quad \exists a \in S\setminus\{0\}: \: ax \le ay.
	\]
\end{lem}

\begin{proof}
	By definition.
\end{proof}

\subsection*{Magnifiable and shrinkable elements}

The following definition gives two relaxed notions of invertibility, which alternatively can be thought of as a type of lower boundedness and upper boundedness, respectively.

\begin{defn}
	Let $S$ be a preordered semiring. An element $x \in S$ is
	\begin{enumerate}
		\item \newterm{magnifiable} if there is $y \in S$ with $xy \ge 1$.
		\item \newterm{shrinkable} if there is $y \in S$ with $xy \le 1$.
	\end{enumerate}
\end{defn}

For example, in a trivially preordered semiring the magnifiable elements are exactly the invertible ones, as are the shrinkable ones.

\begin{lem}
	\label{magnifiable_subsemiring}
	Let $S$ be a preordered semiring with $1 \ge 0$. Then the set of magnifiable elements is closed under addition and multiplication and upwards closed.
\end{lem}

Applying this to $S^\op$ yields a similar statement for shrinkable elements.

\begin{proof}
	Suppose that $x,y \in S$ are magnifiable, so that we have $a,b \in S$ with $xa \ge 1$ and $yb \ge 1$. Then also
	\[
		(xy)(ab) \ge 1,
	\]
	and
	\[
		(x + y)(a + b) \ge xa + yb \ge 2 \ge 1,
	\]
	as was to be shown. Upwards closure is trivial.
\end{proof}

\begin{lem}
	\label{semidomain_zerodivisor}
	Let $S$ be a preordered semiring with $1 > 0$ and such that every nonzero element is magnifiable. Then also:
	\begin{enumerate}
		\item For $x,y \in S$ with $y \neq 0$ there is $a \in S$ with $x \le ay$.
		\item For all $x,y \in S$, we have
			\[
				xy \le 0 \quad \Longrightarrow \quad x = 0 \quad \lor \quad y = 0,
			\]
			and
			\[
				x + y \le 0 \quad \Longrightarrow \quad x = y = 0.
			\]
	\end{enumerate}
\end{lem}

In particular, such $S$ is also a zerosumfree semidomain, and we can therefore form the preordered semifield $\Frac(S)$ by \Cref{fractions_ordered}.

\begin{proof}
	\begin{enumerate}
		\item Upon choosing $b \in S$ with $yb \ge 1$, we have $x \le xby$, so that $a \coloneqq xb$ works. 
		\item We start with the first implication. If $x$ and $y$ were both nonzero, then they would have to be magnifiable. But then so is $xy$, and we get $a$ with $1 \le xya \le 0$, contradicting the assumption $1 > 0$.

			Concerning the second implication, suppose $x + y \le 0$. If $x \neq 0$, then $x$ would have to be magnifiable, so that we have $ax \ge 1$ for suitable $a$. Together with $ay \ge 0$, this gives
			\[
				1 \le 1 + ay \le ax + ay = a(x + y) \le 0,
			\]
			contradicting the assumption $1 > 0$. \qedhere
	\end{enumerate}
\end{proof}

\subsection*{Polynomial growth}

We now consider a condition closely related to magnifiability. It was originally introduced in a more specific stronger form---close to the condition we will consider in \Cref{univ_old}---in~\cite{our_spss}.

\begin{defn}
	\label{univ_defn}
	Let $S$ be a preordered semiring.
	\begin{enumerate}
		\item A \newterm{power universal element} is nonzero $u \in S$ with $u \ge 1$ such that for every nonzero $x,y \in S$ with $x \le y$, there is $k \in \N$ such that
			\[
				y \le x u^k .
			\]
		\item A \newterm{power universal pair} is nonzero $u_-, u_+ \in S$ with $u_- \le u_+$ such that for every nonzero $x,y \in S$ with $x \le y$, there is $k \in \N$ such that
			\[
				y u_-^k \le x u_+^k.
			\]
		\item If $S$ has a power universal pair, then we say that $S$ is \newterm{of polynomial growth}.
	\end{enumerate}
\end{defn}

If $u$ is a power universal element, then taking $u_- \coloneqq 1$ and $u_+ \coloneqq u$ gives a power universal pair, and $S$ is of polynomial growth as well.
We will only work with power universal elements in this paper and reserve the letter ``$u$'' for them.
Power universal pairs plays a role in Part II~\cite{partII} only; we anticipate the definition already now in order to have a definition of polynomial growth consistent across both papers.

\begin{rem}
	Both versions of power universality apply not just for $x \le y$, but automatically under the weaker assumption $x \sim y$.
	This follows from the characterization of $x \sim y$ in terms of the existence of a sequence $z_0, \ldots, z_n$ such that
	$z_i \le z_{i+1}$ or $z_i \ge z_{i+1}$ for all $i$, as well as $x = z_0$ and $z_n = y$, and applying the power universality to every inequality in the sequence.
\end{rem}

The following more particular form of polynomial growth, as introduced in~\cite{our_spss}, covers a wide range of examples.

\begin{lem}
	\label{univ_old}
	If $1 > 0$ in $S$, then an element $u \ge 1$ is power universal if and only if for every nonzero $x \in S$ there is $k \in \N$ with
	\beq
		\label{eq_univ_old}
		x \le u^k , \qquad 1 \le x u^k.
	\eeq
\end{lem}

\begin{proof}
	For the ``if'' part, let $k \in \N$ be large enough so that $y \le u^k$ and $1 \le x u^k$ for given nonzero $x, y \in S$ with $x \le y$. Then also $y \le u^k \le x u^{2k}$, as desired.

	For the ``only if'', applying the assumption to the inequality $1 \le x + 1$, in which both sides are nonzero thanks to $1 \not\le 0$, gives $x+1 \le u^k$ for suitable $k$, which can be weakened to $x \le u^k$. Applying the assumption to this inequality itself gives $\ell \in \N$ with $u^k \le x u^\ell$, which in turn can be weakened to $1 \le x u^\ell$ by $u \ge 1$.
\end{proof}

\begin{ex}
	\label{univ_semifield}
	If $F$ is a preordered semifield, then $F$ is of polynomial growth if and only if there is a power universal element $u$, since a power universal pair can always be replaced by the power universal element $u \coloneqq u_+ u_-^{-1}$.
	Moreover, nonzero $u$ is power universal if and only if $u \ge 1$, and for every nonzero $x \ge 1$ there is $k \in \N$ with $x \le u^k$.
\end{ex}

\begin{ex}
	\label{poly_universal}
	Consider the polynomial semiring of natural number polynomials with strictly positive integer coefficient together with zero,
	\[
		\N_{(+)}[X] \coloneqq \left\{ \sum_i p_i X^i \in \N[X] \:\bigg|\: p_0 > 0 \right\} \cup \{0\},
	\]
	and equipped with the coefficientwise preorder. This is exactly the semiring preorder generated by $1 \ge 0$. Then $u \coloneqq 2 + X$ is a power universal element, as is easily seen by verifying the conditions of \Cref{univ_old}.
\end{ex}

\begin{ex}
	\label{laurent}
	For finitely many variables $\underline{X} = (X_1, \ldots, X_d)$, consider the semiring of Laurent polynomials $\N[\underline{X}, \underline{X}^{-1}]$ with respect to the coefficientwise preorder. Then
	\[
		u \coloneqq 1 + \sum_i (X_i^{-1} + X_i)
	\]
	is a power universal element. More generally, if $S$ is any preordered semiring with $1 > 0$ and having a power universal element $v$, then $S[X, X^{-1}]$ is a preordered semiring with respect to the coefficientwise preorder, and has a power universal element given by
	\[
		u \coloneqq v + X^{-1} + X.
	\]
\end{ex}

\begin{ex}
	\label{polysos}
	In the polynomial ring $\R[\underline{X}] = \R[X_1,\ldots,X_d]$, consider the subsemiring $\Sigma_+ \subseteq \R[\underline{X}]$ consisting of all sums of squares plus constants, as in \eqref{sigma}. For $p,q \in \Sigma_+$, we put $p \le q$ if $q - p$ is itself a sum of squares. This makes $\Sigma_+$ into a preordered semiring with $1 > 0$. It is of polynomial growth with respect to $u \coloneqq 2 + \sum_i X_i^2$.

	More generally, let $R$ be a ring and $\Sigma_{+,R}$ as in \eqref{sigmaR}. Then $\Sigma_{+,R}$ becomes a preordered semiring with respect to the preorder in which $x \le y$ if and only if $y - x$ is itself a sum of squares. If $R$ is finitely generated by $a_1,\ldots,a_n \in R$, then $u \coloneqq 2 + \sum_i a_i^2$ is a power universal element.
\end{ex}

In fact, this example is an instance of the following more general observation.

\begin{lem}[{\cite{our_spss}}]
	\label{polygrowth_2chars}
	Let $S$ be a preordered semiring with $1 > 0$. Suppose that there is $v \in S$ such that for every nonzero $x \in S$ there is $p \in \N[X]$ with
	\beq
		\label{boundp}
		x \le p(v), \qquad 1 \le p(v) x.
	\eeq
	Then $S$ is of polynomial growth with respect to the power universal element $u \coloneqq 2 + v$.
\end{lem}

\begin{proof}
	\Cref{poly_universal} shows that every such polynomial $p(v)$ can be upper bounded by a power of $2 + v$. Hence the claim follows by \Cref{univ_old}.
\end{proof}

\begin{rem}
	\label{univ_gen}
	As for magnifiable elements in \Cref{magnifiable_subsemiring}, the set of elements $x$ which satisfy the polynomial growth condition in the form \eqref{eq_univ_old} or \eqref{boundp} is closed under addition and multiplication. It follows that these conditions only need to be verified on some subset which generates $S$.
\end{rem}

\begin{rem}
	\label{nonzeros}
	If $S$ has a power universal element $u$ and satisfies $1 > 0$, then every nonzero element is magnifiable.
	Thus \Cref{semidomain_zerodivisor} applies, and we can conclude in particular that $S$ is a zerosumfree preordered semidomain. Moreover, $\Frac(S)$ is then a preordered semifield of polynomial growth.
\end{rem}

We end with an important example in which a power universal element does not exist, but a power universal pair does.

\begin{ex}
	Let $\N[\underline{X}]$ be the polynomial semiring in one variable equipped with the coefficientwise preorder. Then a power universal element does not exist already for one variable: there is clearly no polynomial $u$ and no $k \in \N$ such that $(1 + X) \le X u^k$.

	However, $\N[\underline{X}]$ does have a power universal pair given by
	\[
		u_- \coloneqq \prod_i X_i, \qquad u_+ \coloneqq u_- + \sum_i (1 + X_i^2) \prod_{j\neq i} X_j.
	\]
	The fact that these work follows from the case of Laurent polynomials considered in \Cref{laurent}, since the power universal pair $(1,u)$ given there produces the current one upon multiplication by $\prod_i X_i$, which makes sure that the negative exponents in $u$ are eliminated.
\end{ex}

\section{Multiplicatively Archimedean totally preordered semifields}
\label{malt_total}

The results of this section generalize the classical fact that the real numbers are the only Dedekind complete totally ordered Archimedean field. In our semifield setting, another such model object crops up: the tropical reals from \Cref{basic_semifields}.

Technically, we now discuss a notion of Archimedeanicity for preordered semifields. Although the following terminology unfortunately clashes somewhat with the established notion of Archimedeanicity in real algebra, we have chosen the term \emph{multiplicatively Archimedean} since our definition essentially says that the multiplicative group must be Archimedean in the sense in which the term is standardly used for preordered abelian groups (see e.g.~\cite[Chapter~4]{glass}). We routinely use that a preordered semifield is of polynomial growth if and only if it has a power universal element (\Cref{univ_semifield}).

\begin{defn}
	\label{march}
	A preordered semifield $F$ is \newterm{multiplicatively Archimedean} if $x^k \le y$ for all $k \in \N$ implies $x \le 1$ for all $x, y \in F^\times$.
\end{defn}

In the totally preordered case, this equivalently states that every $x \in F^\times$ with $x > 1$ is a power universal element: for every nonzero $y \ge 1$ in $F$ there is $k \in \N$ such that $y \le x^k$, or even $y < x^k$ by increasing $k$.

\begin{thm}
	\label{real_or_tropical}
	Let $F$ be a totally preordered semifield. Then $F$ is multiplicatively Archimedean if and only if it order embeds into one of the following preordered semifields:
	\[
		\R_+, \qquad \R_+^\op, \qquad \TR_+, \qquad \TR_+^\op.
	\]
\end{thm}

These four cases are mutually exclusive: an embedding into $\R_+$ or $\TR_+$ requires $1 \ge 0$, while an embedding into their opposites requires $1 \le 0$. In the former case, an embedding into $\R_+$ requires $1 + 1 > 1$, while an embedding into $\TR_+$ requires $1 + 1 \approx 1$.

\begin{proof}
	We assume $1 \ge 0$ in $F$ without loss of generality, and then show that $F$ is multiplicatively Archimedean if and only if it order embeds into $\R_+$ or $\TR_+$. Since the latter preordered semifields are clearly multiplicatively Archimedean, the ``if'' direction is done.

	For the ``only if'' part, consider the multiplicative group $F^\times$ as a totally preordered group. Since it is Archimedean by assumption, a classical result of H\"older~\cite[Theorem~XIII.12]{birkhoff} implies that there is an order embedding of ordered abelian groups $\alpha : F^\times \to (\Rplus,\cdot)$. Here we use the multiplicative group $(\Rplus,\cdot)$ instead of the usual additive group $(\R,+)$ for the target of the embedding, which is possible since these two are isomorphic, and multiplicativity is what we want.

	We have $n \not\approx 0$ in $F$ for all $n \in \Nplus$ (\Cref{approx_zero}). We now distinguish two cases:
	\begin{itemize}
		\item We have $m \approx n$ in $F$ for two distinct natural numbers $m$ and $n$. 

			Assuming that $m > n$ in $\N$ without loss of generality, it follows that $n + 1 \approx n$ in $F$. An induction argument then shows that all natural numbers above $n$ are $\approx$-equivalent to $n$. We therefore obtain $1 = \frac{n}{n} \approx \frac{2n}{n} = 2$. But this makes addition in $F$ idempotent up to $\approx$, in the sense that $x + x \approx x$. Hence addition is equal to joins with respect to the underlying total preorder: we trivially have $x,y \le x+y$; and if $x,y \le z$, then also
			\[
				x + y \le z + z \approx z.
			\]
			Therefore $x + y$ is the join of $x$ and $y$, or equivalently $x + y \approx \max(x,y)$.
		
			Then since $\alpha(\max(x,y)) = \max(\alpha(x),\alpha(y))$ holds automatically, $\alpha$ is already a semiring homomorphism $F \to \TR_+$ when extended from $F^\times$ to $F$ via $\alpha(0) \coloneqq 0$. This is the desired order embedding.
		\item We have $m \not\approx n$ for all distinct natural numbers $m$ and $n$.

			Now the inclusion $\Q_+ \subseteq F$ is an order embedding.
			Restricting $\alpha$ to $\Qplus$ gives an order embedding $\Qplus \to \Rplus$ between multiplicative groups. By the monotone case of the Cauchy functional equation, it is therefore of the form $q \mapsto q^s$ for some exponent $s \in [0,\infty)$, where in our case $s \neq 0$ since it is an order embedding. Upon replacing $\alpha$ by $\alpha^{1/s}$, we can therefore assume without loss of generality that $\alpha(q) = q$ for every $q \in \Qplus$.

			It remains to be shown that such $\alpha$ is additive. To this end, for given $x \in F^\times$ consider the elementary fact that every real number is the supremum of its rational lower bounds and the infimum of its rational upper bounds,
			\[
				\alpha(x) = \sup \,\left\{q \in \Qplus \mid q \le \alpha(x)\right\} = \inf \,\left\{q \in \Qplus \mid q \ge \alpha(x)\right\}.
			\]
			Since $\alpha$ is an order embedding and $q = \alpha(q)$ for every $q \in \Qplus$, the condition $q \le \alpha(x)$ in $\Rplus$ is equivalent to $q \le x$ in $F$, so that
			\[
				\alpha(x) = \sup \,\left\{q \in \Q_+ \mid q \le x\right\} = \inf \,\left\{q \in \Q_+ \mid q \ge x\right\}.
			\]
			Using the supremum formula gives, for all $x, y \in F^\times$
			\begin{align*}
				\alpha(x + y)	& = \sup \,\left\{q \in \Qplus \mid q \le x + y\right\} \\[4pt]
				& \ge \sup \,\left\{q_x \in \Qplus \mid q_x \le x\right\} + \sup \,\left\{q_y \in \Qplus \mid q_y \le y\right\} \\[4pt]
				& = \alpha(x) + \alpha(y),
			\end{align*}
			while the other inequality $\alpha(x + y) \le \alpha(x) + \alpha(y)$ follows similarly from the infimum formula. \qedhere
	\end{itemize}
\end{proof}

\begin{cor}
	\label{arch_truncate}
	Let $F$ be a totally preordered semifield of polynomial growth. Then there is a monotone homomorphism $\phi : F \to \mathbb{K}$ for some $\mathbb{K} \in \{\R_+, \R_+^\op, \TR_+, \TR_+^\op\}$ such that $\phi(u) > 1$ for every power universal element $u > 1$.
\end{cor}

\begin{proof}
	Fix a power universal element $u > 1$.\footnote{This can be assumed to hold with strict inequality, since $u \approx 1$ is power universal only if all nonzero elements are $\approx$-equivalent, which makes $F$ equivalent to $\B$ or $\B^\op$, for which the statement is trivially true.}
	Then the new preorder $\le_0$ defined as
	\[
		x \le_0 y \qquad : \Longleftrightarrow \qquad x^n \le y^n u \quad \forall n \in \Nplus,
	\]
	makes $F$ again into a totally preordered semifield, and such that $\le_0$ extends $\le$. In particular this preordered semifield is again totally preordered. It is multiplicatively Archimedean since $x >_0 1$ implies that there is $n \in \Nplus$ with $x^n > u$, making $x$ power universal as well.
	Hence \Cref{real_or_tropical} delivers a homomorphism $\phi : F \to \mathbb{K}$ that is an order embedding with respect to $\le_0$, and in particular $\le$-monotone.
	Also $\phi(u) > 1$ since $u >_0 1$.
	If $v$ is any other power universal element, then we must have $v^k \ge u$ for some $k \in \N$, and therefore also $\phi(v) > 1$.
\end{proof}

We also call the preordered semifield $(F,\le_0)$ constructed in the proof the \newterm{multiplicatively Archimedean truncation} of $F$.

While these results are useful criteria for the existence of an order embedding, 
we can further classify multiplicatively Archimedean totally preordered semifields up to equivalence if we also assume a suitable form of Dedekind completeness. The following definition is a variant of the standard one for ordered fields. We assume $1 \ge 0$ for simplicity and without loss of generality.

\begin{defn}
	Let $F$ be a totally preordered semifield with $1 \ge 0$. Then a \newterm{Dedekind cut} in $F$ is a subset $L \subseteq F$ such that:
	\begin{enumerate}[label=(\roman*)]
		\item $L$ is downward closed.
		\item If $x \in L$, then there is $y \in L$ with $y > x$.
		\item $L \neq F$.
	\end{enumerate}
\end{defn}

We associate to every $x \in F$ the Dedekind cut
\[
	L_x \coloneqq \{y \in F \mid y < x\},
\]
and in particular $L_0 = \emptyset$. We say that $F$ is \newterm{Dedekind complete} if every Dedekind cut is equal to some $L_x$.

By analogy with the classical result that $\R$ is the only Dedekind complete Archimedean totally ordered field, one might hope that if a multiplicatively Archimedean totally preordered semifield $F$ is in addition Dedekind complete, then $F / \!\approx$ would be isomorphic to $\R_+$ or to $\TR_+$. However, this is not the case, as $\TR_+$ has Dedekind complete subsemifields, namely the Boolean semifield $\B$ or more generally the collection of integer powers of any fixed nonzero element (together with $0$). But these are be the only counterexamples.
	
\begin{thm}
	\label{dedekind_class}
	If $F$ is a Dedekind complete and multiplicatively Archimedean totally preordered semifield with $1 \ge 0$, then $F$ is equivalent to exactly one of the following:
	\[
		\R_+, \qquad \B, \qquad \TZ_+, \qquad \TR_+.
	\]
	Dually, with $1 \le 0$ instead, we obtain the four possibilities $\R_+^\op$, $\B^\op$, $\TZ_+^\op$ and $\TR_+^\op$.
\end{thm}

\begin{proof}
	By \Cref{real_or_tropical}, it is enough to show that the above list contains exactly the Dedekind complete subsemifields of $\R_+$ and $\TR_+$. For $\R_+$, every subsemifield must contain $\Q_+$. Since every nonnegative real number defines a Dedekind cut in $F$ and these Dedekind cuts are all different, $\R_+$ itself is the only Dedekind complete subsemifield.
	
	For $\TR_+$, a Dedekind complete subsemifield in particular contains a subgroup of the multiplicative group $\TR_+^\times$. In terms of the additive picture of $\TR_+$, we use the standard fact that a subgroup of $(\R,+)$ is singly generated or dense. Hence $F^\times$ for a given Dedekind complete subsemifield $F \subseteq \TR_+$ either consists of powers of a given element, which makes it isomorphic to $\B$ or to $\TZ_+$, or it is all of $\TR_+$ by an argument analogous to the previous paragraph in the dense case.
\end{proof}

\section{A peculiar polynomial identity}
\label{curioussec}

We now turn to the derivation of some deeper results. 
This starts with a polynomial identity which may on first look seem unrelated to our overarching theme, but will turn out to be a central ingredient in the proof of our abstract Vergleichsstellensatz for preordered semifields (\Cref{abstract_semifield_pss}).

\begin{lem}
	\label{curiouslem}
	For $n \in \N$, let $\underline{A} = \left(A_0,\ldots,A_n\right)$ and $\underline{B} = \left(B_0,\ldots,B_n\right)$ be finite sequences of variables. Then in the semiring $\N[\underline{A},\underline{B},X,Y]$, we have
	\begin{align*}
		& \sum_{i=0}^n \: \left( A_i \sum_{j=0}^n B_j Y^j + B_i \sum_{j=0}^n A_j X^j \right) \left( \sum_{k=1}^i X^{i-k} Y^{k-1} \right) \\[8pt]
		= & \: \sum_{i=0}^n \: \left( A_i \sum_{j=0}^n B_j X^j + B_i \sum_{j=0}^n A_j Y^j \right) \left( \sum_{k=1}^i X^{i-k} Y^{k-1} \right).
	\end{align*}
\end{lem}

As our notation suggests, we find it useful to think of $X$ and $Y$ as the primary polynomial variables and of the $\underline{A}$ and $\underline{B}$ as playing the role of coefficients, but formally the latter are also just variables. Note that the two sides of the equation are identical except for $Y^j$ in the first factor on the left replaced by $X^j$ on the right, and vice versa.

\begin{proof}
	Reindexing shows that the second factor on each side is symmetric in $X$ and $Y$.
	It follows that the two sides of the equation only differ by exchanging $X$ and $Y$.
	Hence we need to show that the left-hand side is invariant under $X \leftrightarrow Y$.

	For any $i,j \in \{0,\ldots,n\}$ and $\ell,m \in \{0,\ldots,2n-1\}$, we count how many times the monomial $A_i B_j X^\ell Y^m$ occurs on the left-hand side. Multiplying out and inspecting shows that the multiplicity of this term is given by
	\begin{align}
		\begin{split}
			+1 & \qquad \text{if } \ell + m = i + j - 1 \text{ and } 0 \le \ell \le i - 1, \\
			+1 & \qquad \text{if } \ell + m = i + j - 1 \text{ and } 0 \le m \le j - 1, \\
		\end{split}
	\end{align}
	where these two conditions are mutually exclusive. Since the bounds on $m$ in the second condition are equivalent to $i \le \ell \le i + j - 1$, exactly one of the two inequality conditions holds as soon as the equation $\ell + m = i + j - 1$ holds. Hence the left-hand side is equal to
	\[
		\sum_{i,j=0}^n A_i B_j \sum_{\ell,m \in \N \:\mid\: \ell + m = i + j - 1} X^\ell Y^m.
	\]
	This is indeed manifestly symmetric under $X \leftrightarrow Y$.
\end{proof}

\section{An abstract Vergleichsstellensatz for preordered semifields}
\label{semifields}

Throughout this section, $F$ is any preordered semifield.

\begin{lem}
	\label{fourier}
	Suppose that elements $r_1,\ldots,r_n,x,y \in F$ satisfy
	\[
		\sum_{i=1}^n r_i x^i \le \sum_{i=1}^n r_i y^i,
	\]
	where $r_i \neq 0$ for at least one $i$. Then $x \le y$.
\end{lem}

\begin{proof}
	The claim is trivial if $x = y = 0$. Upon reversing the order if necessary, we can therefore assume without loss of generality that $x$ is nonzero and hence invertible. Then 
	\begin{align}
		\begin{split}
		\label{overallineq}
			x \sum_{j=1}^n \sum_{i=j}^n r_i x^{j-1} y^{i-j}		& = \sum_{j=1}^n r_j x^j + \sum_{j=1}^n \sum_{i=j+1}^n r_i x^j y^{i-j} \\[4pt]
										& \le \sum_{i=1}^n r_i y^i + \sum_{j=2}^n \sum_{i=j}^n r_i x^{j-1} y^{i-j+1} \\[4pt]
										& = y \sum_{j=1}^n \sum_{i=j}^n r_i x^{j-1} y^{i-j},
		\end{split}
	\end{align}
	where the inequality step uses the assumption. Since $x$ is invertible and some $r_i$ is invertible, it follows that also some term $r_i x^{j-1} y^{i-j}$ is invertible, namely in particular a suitable $j = i$ one. Therefore also the entire expression $\sum_{j=1}^n \sum_{i=j}^n r_i x^{j-1} y^{i-j}$ is invertible. Dividing the overall inequality \eqref{overallineq} by this expression gives the desired $x \le y$.
\end{proof}

\begin{rem}
	\label{geom_fourier}
	As a special case of \Cref{fourier}, we have the perhaps surprising equivalence that for any $n \in \Nplus$,
	\beq
		\label{depower}
		x \le y \qquad \Longleftrightarrow \qquad x^n \le y^n.
	\eeq
	It is instructive to spell out the relevant argument separately. For the nontrivial implication, the assumption $x^n \le y^n$ gives
	\[
		x \sum_{i=1}^n x^{i-1} y^{n-i} = \sum_{i=1}^n x^i y^{n-i} \le \sum_{i=1}^n x^{i-1} y^{n-i+1} = y \sum_{i=1}^n x^{i-1} y^{n-i},
	\]
	and hence $x \le y$ upon dividing by the sum, which is nonzero as soon as $x$ or $y$ is.
\end{rem}

If $x,y \in F$ satisfy $x \not\le y$, then by \Cref{generate_order} there is a smallest semifield preorder $\preceq$ extending the given $\le$ in the sense that $a \le b$ implies $a \preceq b$ and moreover such that $x \preceq y$. We call $\preceq$ the semifield preorder obtained by \newterm{adjoining} $x \le y$. We establish an explicit characterization of this preorder.

\begin{lem}
	\label{adjoin}
	Let $x,y \in F$, and let $\preceq$ denote the semiring preorder obtained by adjoining $x \le y$. Then:
	\begin{enumerate}
		\item For any $a,b \in F$, we have $a \preceq b$ if and only if there is a finite sequence $r_0,\ldots,r_n \in F$ such that
			\beq
				\label{adjoin_ineqs}
				a \le \sum_{i=0}^n r_i x^i, \qquad \sum_{i=0}^n r_i y^i \le b.
			\eeq
		\item If $y \not\le x$, then also $y \not\preceq x$ and $\preceq$ makes $F$ into a preordered semifield.
	\end{enumerate}
\end{lem}

Note that in contrast to \Cref{fourier}, the sums now start at $i = 0$. 

\begin{proof}
	\begin{enumerate}
		\item Let us tentatively write $\preceq$ for the relation defined by \eqref{adjoin_ineqs}. Then it is straightforward to see that $\preceq$ respects addition and multiplication, in the sense that $a \preceq b$ implies $c + a \preceq c + b$ and $ca \preceq cb$ for any $c \in F$. In order to see that $\preceq$ defines a semiring preorder, only transitivity is missing.

			Thus assume that we have $a,b,c\in F$ and finite sequences $(r_i)_{i=0}^m$ and $(s_i)_{i=0}^n$ which witness $a \preceq b \preceq c$,
			\[
				a \le \sum_{i=0}^m r_i x^i, \qquad \sum_{i=0}^m r_i y^i \le b \le \sum_{i=0}^n s_i x^i, \qquad \sum_{i=0}^n s_i y^i \le c.
			\]
			We then first prove that $a b \preceq c b$. Upon padding both sequences by zeroes as far as necessary, define the sequence $(t_i)_{i=0}^{m+n}$ as their convolution,
			\[
				t_j \coloneqq \sum_{i=0}^j r_i s_{j-i} \qquad \forall j = 0, \ldots, m+n.
			\]
			This gives
			\[
				a b \le \left( \sum_i r_i x^i \right) \left( \sum_j s_j x^j \right) = \sum_i t_i x^i,
			\]
			where all sums extend as far as necessary, and similarly
			\[
				\sum_i t_i y^i = \left( \sum_i r_i y^i \right) \left( \sum_j s_j y^j \right) \le b c,
			\]
			so that we can conclude the claimed $a b \preceq c b$. Now if $b \neq 0$, then monotonicity of multiplication by $b^{-1}$ gives the desired $a \preceq c$. For $b = 0$, we argue differently, using
			\[
				a \le \sum_i (r_i + s_i) x^i, \qquad \sum_i (r_i + s_i) y^i \le c,
			\]
			which again gives $a \preceq c$.

			It is obvious that $\preceq$ extends $\le$ and in addition satisfies $x \preceq y$.
			Finally, it remains to prove that $\preceq$ has the claimed minimality property. Thus suppose that $\leqslant$ is any other semiring preorder which extends $\le$ and satisfies $x \leqslant y$, and assume $a \preceq b$ with witnessing sequence $(r_i)$. Then we get $x^i \leqslant y^i$ by $x \leqslant y$ for any $i \in \N$, and therefore
			\[
				a \leqslant \sum_i r_i x^i \leqslant \sum_i r_i y^i \leqslant b,
			\]
			since $\leqslant$ extends $\le$. This gives the desired $a \leqslant b$.
		\item Suppose that we had $y \preceq x$. This would mean that we had a finite sequence $(r_i)_{i=0}^n$ such that 
			\[
				y \le \sum_{i=0}^n r_i x^i, \qquad  \sum_{i=0}^n r_i y^i \le x.
			\]
			Using this in the form
			\[
				\sum_{i=0}^n r_i y^{i+1} \le x y \le \sum_{i=0}^n r_i x^{i+1}
			\]
			and reindexing lets us deduce the claim $y \le x$ from \Cref{fourier} if $r_i \neq 0$ for some $i$. If all $r_i = 0$, then $y \le 0 \le x$ trivially implies the desired $y \le x$.

			To show that $\preceq$ is still a semifield preorder, by \Cref{approx_zero} it remains to note that $\preceq$ is not complete, which we have just shown.
			\qedhere
	\end{enumerate}
\end{proof}

\begin{lem}
	\label{notboth}
	Let $a, b, x, y \in F$. Suppose that adjoining $x \le y$ results in $a \preceq b$, as does adjoining $y \le x$. Then already $a \le b$.
\end{lem}

\begin{proof}
	Since adjoining $x \le y$ results in $a \preceq b$, \Cref{adjoin} gives $r_0,\ldots,r_n \in F$ such that
	\[
		a \le \sum_{j=0}^n r_j x^j, \qquad \sum_{j=0}^n r_j y^j \le b.
	\]
	Similarly since adjoining $y \le x$ results in $a \preceq b$ as well, we also obtain $s_0,\ldots,s_n \in F$ such that
	\[
		a \le \sum_{j=0}^n s_j y^j, \qquad \sum_{j=0}^n s_j x^j \le b,
	\]
	where by padding with zeroes we have assumed that the sequences $(r_j)$ and $(s_j)$ have the same length. We treat some degenerate cases first. If $x = y = 0$, then trivially also $a \le b$ since then $\preceq$ coincides with $\le$. If $r_i = s_i = 0$ for all $i \ge 1$, then the above inequalities give $a \le r_0 \le b$ and hence also $a \le b$. We thus can assume that $x \neq 0$ or $y \neq 0$, and $r_i \neq 0$ or $s_i \neq 0$ for some $i \ge 1$.
	
	We now combine the previous inequalities with the peculiar polynomial identity of \Cref{curiouslem},
	\begin{align*}
		a \sum_{i=0}^n & \left( r_i + s_i \right) \left( \sum_{k=1}^i x^{i-k} y^{k-1} \right) \\[8pt]
			& \le \sum_{i=0}^n \left( r_i \sum_{j=0}^n s_j y^j + s_i \sum_{j=0}^n r_j x^j \right) \left( \sum_{k=1}^i x^{i-k} y^{k-1} \right) \\[8pt]
			& = \sum_{i=0}^n \left( r_i \sum_{j=0}^n s_j x^j + s_i \sum_{j=0}^n r_j y^j \right) \left( \sum_{k=1}^i x^{i-k} y^{k-1} \right) \\[8pt]
			& \: \le b \sum_{i=0}^n \left( r_i + s_i \right) \left( \sum_{k=1}^i x^{i-k} y^{k-1} \right) .
	\end{align*}
	This implies the claim $a \le b$ upon cancelling, since the non-degeneracy assumptions from the previous paragraph together with $F$ being a strict semifield imply that the factor $\sum_{i=0}^n (r_i + s_i) \sum_{k=1}^i x^{i-k} y^{k-1}$ is nonzero. 
\end{proof}

This puts us in a position to derive the core ingredient of our abstract Vergleichsstellensatz.

\begin{prop}
	\label{key}
	Suppose that $x,y \in F$ are such that $x \not\le y$ and that the preorder on $F$ is maximal with this property. Then $F$ is totally preordered.
\end{prop}

\begin{proof}
	Suppose that $a \not\le b$ and $b \not\le a$ for two arbitrary $a,b \in F$. From this we derive a contradiction with the maximality assumption. Since $F$ carries the maximal semifield preorder which satisfies $x \not\le y$, adjoining $a \le b$ must result in $x \preceq y$, and similarly adjoining $b \le a$ must also give $x \preceq y$. But then \Cref{notboth} implies $x \le y$, which contradicts the assumption $x \not\le y$.
\end{proof}

We can now state and prove our \newterm{abstract Vergleichsstellensatz for preordered semifields}.

\begin{thm}
	\label{abstract_semifield_pss}
	Let $F$ be a preordered semifield. Then its preorder $\le$ is the intersection of all total semifield preorders on $F$ which extend it.
\end{thm}

In other words, for $x,y \in F$ we have $x \le y$ if and only if $x \preceq y$ for every total semifield preorder $\preceq$ on $F$ which extends the given preorder $\le$.

\begin{proof}
	If $x \le y$, then trivially also $x \preceq y$ by the definition of extension. So suppose that $x \not\le y$. Then by Zorn's lemma, we can extend $\le$ to a semifield preorder $\preceq$ which still satisfies $x \not\preceq y$ and is maximal with this property. The claim now follows from \Cref{key}.
\end{proof}

\Cref{abstract_semifield_pss} is interesting already in the case where $F$ is trivially preordered, as follows.

\begin{cor}
	Let $F$ be a strict semifield and $x \neq y$ in $F$. Then there is a total semifield preorder $\le$ on $F$ such that $x < y$.
\end{cor}

\begin{proof}
	Apply \Cref{abstract_semifield_pss} with the trivial preorder on $F$.
\end{proof}

\begin{rem}
	\label{semifield_possible}
	In contrast to the situation with fields and ordered fields, every strict semifield has \emph{some} total semifield preorder.
\end{rem}

\section{An abstract Vergleichsstellensatz for preordered semirings}
\label{n1}

We now turn to the consideration of preordered semirings in order to see what conclusions we can draw about these from our previous results on preordered semifields.
The theorem we develop can be considered a non-Archimedean generalization of the classical Positivstellensatz of Krivine--Kadison--Dubois, which we then rederive as a special case.

Throughout this section, $S$ is a preordered semiring with $1 > 0$ and power universal element $u \in S$. Recall that such $S$ is automatically a zerosumfree semidomain (\Cref{nonzeros}).

\subsection*{The test spectrum}

We will have to make some topological considerations on the set of monotone homomorphisms into $\R_+$ and $\TR_+$, or equivalently into arbitrary multiplicatively Archimedean totally preordered semifields.
In the following, we use the multiplicative picture of $\TR_+$.

\begin{defn}
	The \newterm{test spectrum} of $S$ is the disjoint union
	\begin{align*}
		\Sper{S} \: \coloneqq \:	& \{\text{\normalfont{monotone homs }} S \to \R_+ \} \\
					& \sqcup \{\text{\normalfont{monotone homs }} S \to \TR_+ \text{\normalfont{ with }} \phi(u) = e \},
	\end{align*}
	where the first piece is the \newterm{real part} and the second the \newterm{tropical part}, respectively.
\end{defn}

Since exponentiation by any fixed positive real is an automorphism of $\TR_+$, and since $\phi(u) \ge 1$ for any $\phi : S \to \TR_+$, the requirement $\phi(u) = e$ is a choice of normalization. The Euler constant $e$ here could be replaced by any other real $>1$, in which case also the base of the logarithms that we use in the following will have to be replaced accordingly. Fixing the normalization like this also excludes the trivial monotone homomorphism $S \to \TR_+$ given by the composite $S \to \B \to \TR_+$, which maps every nonzero element to $1 \in \TR_+$, from the test spectrum.

\begin{lem}
	\label{phiu}
	For every $\phi \in \Sper{S}$ we have $\phi(u) > 1$.
\end{lem}

\begin{proof}
	This holds for $\phi : S \to \TR_+$ by definition. For $\phi : S \to \R_+$ we have $\phi(2) = 2$. Choosing $k \in \Nplus$ with $2 \le u^k$ shows that $\phi(u) = \phi(u^k)^{1/k} \ge \phi(2)^{1/k} = 2^{1/k} > 1$.
\end{proof}

We therefore have $\log \phi(u) \neq 0$, guaranteeing in particular that the denominator in the following definition is nonzero.

\begin{defn}
	\label{1test_topology}
	$\Sper{S}$ carries the weakest topology which makes the \newterm{logarithmic evaluation maps}
	\beq
		\label{ev_map}
		\lev_x \: : \: \Sper{S} \longrightarrow [0,\infty), \qquad \phi \longmapsto \frac{\log \phi(x)}{\log \phi(u)}
	\eeq
	continuous for all nonzero $x \in S$.
\end{defn}

The denominator has been chosen such that $\lev_u(\phi) = 1$ for all $\phi \in \Sper{S}$.
Note also that if $x \le u^k$ and $1 \le x u^k$, then $\lev_x$ takes values in $[-k,+k]$.

\begin{rem}
	The tropical part of $\Sper{S}$ is a closed subset, since it contains exactly all of those $\phi$ which satisfy $\lev_2(\phi) = 0$.
	Indeed for $\phi : S \to \TR_+$ we have $\phi(2) = \phi(1) + \phi(1) = 1$, while for $\phi : S \to \R_+$ we have $\phi(2) = 2$.
\end{rem}

\begin{rem}
	The topology of $\Sper{S}$ does not depend on the choice of power universal element $u$. For if $v \in S$ is power universal as well, then the logarithmic evaluation maps associated to $v$ can be written in terms of the ones of $u$ as $\lev_x - \lev_v$ for any nonzero $x$, which is continuous as a difference of two continuous functions, and likewise the other way around.
\end{rem}

Upon equipping the set of continuous functions $C(\Sper{S})$ with the pointwise algebraic structure and pointwise order, we can consider it as a partially ordered ring.
We record two simple properties relating to this structure.

\begin{lem}
	\label{lev_hom}
	For any nonzero $x,y \in S$, we have:
	\begin{enumerate}
		\item If $x \le y$, then also $\lev_x \le \lev_y$.
		\item $\lev_{xy} = \lev_x + \lev_y$.
	\end{enumerate}
\end{lem}

\begin{proof}
	Straightforward.	
\end{proof}

For many $S$, it is possible to find a sequence of monotone homomorphisms $\phi : S \to \R_+$ which converges to a monotone homomorphism $\phi : S \to \TR_+$ in $\Sper{S}$. This is closely related to the standard construction of $\TR_+$ as the tropical limit~\cite{tropical,LM}, also know as \newterm{Maslov dequantization}. In our setting, Maslov dequantization amounts to the following elementary facts of which we omit the proof.

\begin{lem}
	\label{maslov}
	For $\alpha,\beta \in (-\infty,\infty)$, we have
	\begin{align*}
		\max(e^\alpha, e^\beta)	& = \lim_{r \to \infty} \left(e^{r\alpha} + e^{r\beta}\right)^{1/r},	\\
		\max(\alpha, \beta) \,	& = \lim_{r \to \infty} r^{-1} \log \left(e^{r\alpha} + e^{r\beta} \right).
	\end{align*}
	and both limits are uniform on compact subsets in the parameters $\alpha$ and $\beta$.
\end{lem}

\begin{lem}
	\label{philevadd}
	For nonzero $x,y \in S$ and all real $\phi \in \Sper{S}$, we have
	\beq
		\label{expadd}
		2^{\frac{\lev_{x + y}(\phi)}{\lev_2(\phi)}} = 2^{\frac{\lev_x(\phi)}{\lev_2(\phi)}} + 2^{\frac{\lev_y(\phi)}{\lev_2(\phi)}}.
	\eeq
\end{lem}

\begin{proof}
	Straightforward computation.
\end{proof}

For us, the main reason for normalizing the logarithmic evaluation maps as in~\eqref{ev_map} is to ensure the following:

\begin{prop}
	\label{chaus}
	$\Sper{S}$ is a compact Hausdorff space.	
\end{prop}

\begin{proof}
	First an auxiliary statement: using \Cref{maslov}, it is elementary to show that the set $\mathcal{A} \subseteq \R^4$ consisting of all $(\alpha,\beta,\gamma,\delta) \in \R^4$ such that
	\[
		\left[ (\delta > 0) \:\land\: (2^{\gamma/\delta} = 2^{\alpha/\delta} + 2^{\beta/\delta} ) \right] \quad \lor \quad
		\left[ (\delta = 0) \:\land\: (\gamma = \max(\alpha,\beta)) \right]
	\]
	is closed in the Euclidean topology.

	For any nonzero $x \in S$, choose $k_x \in \N$ such that $x \le u^{k_x}$ and $1 \le u^{k_x} x$. Then for all $\phi \in \Sper{S}$,
	\[
		-k_x \le \lev_x(\phi) \le k_x.
	\]
	Thus by definition of the topology, $\Sper{S}$ is a subspace of the compact Hausdorff space $\prod_{x \in S \setminus \{0\}}\, [-k_x,k_x]$. To show that it is closed, it is enough to characterize this subset in terms of closed conditions. We do this by noting that a point $(\nu_x) \in \prod_{x \in S \setminus \{0\}} [-k_x, k_x]$ belongs to $\Sper{S}$ if and only it satisfies the following conditions:
	\begin{enumerate}
		\item $\nu_u = 1$.
		\item $\nu_{xy} = \nu_x + \nu_y$ for all nonzero $x, y \in S$.
		\item $\nu_x \le \nu_y$ for all nonzero $x, y \in S$ with $x \le y$.
		\item $(\nu_x, \nu_y, \nu_{x + y}, \nu_2) \in \mathcal{A}$ for all nonzero $x, y \in S$.
	\end{enumerate}
	Indeed, we have already observed that every $\nu_x \coloneqq \lev_x(\phi)$ satisfies these conditions. Conversely, suppose that $\nu$ satisfies these conditions. If $\nu_2 = 0$, then $\phi(x) \coloneqq e^{\nu_x}$ defines a tropical spectral point. While if $\nu_2 \neq 0$, then $\phi(x) \coloneqq 2^{\nu_x / \nu_2}$ defines a real spectral point.
\end{proof}

\begin{rem}
	\label{tsper_functorial}
	$\Sper{-}$ is functorial as follows. If $f : S \to T$ is a monotone homomorphism of preordered semirings, both of which satisfy our assumptions and such that $f(u) \in T$ is a power universal element of $T$, then composition with $f$ defines a continuous map
	\[
		\Sper{T} \longrightarrow \Sper{S}.
	\]
\end{rem}

\begin{ex}
	\label{polyspectrum}
	Consider the semiring of single-variable polynomials $\N_{(+)}[X]$ with strictly positive constant coefficient (together with the zero polynomial) and with respect to the coefficientwise preorder. Using $u = 2 + X$ as in \Cref{poly_universal}, we can identify the test spectrum with the compact Hausdorff space $[0,\infty]$ as follows:
	\begin{enumerate}
		\item\label{pointeval} The monotone homomorphisms $\N_{(+)}[X] \to \R_+$ are the evaluation maps $p \mapsto p(r)$, parametrized by $r \in [0,\infty)$.
		\item\label{degree} Using the additive picture of $\TR_+$, there is exactly one normalized monotone homomorphism $\N_{(+)}[X] \to \TR_+$, namely the degree map $p \mapsto \deg(p)$.
		\item\label{sper} Identifying the degree map with $\infty$ produces a homeomorphism
			\[
				\Sper{\N_{(+)}[X]} \cong [0,\infty].
			\]
	\end{enumerate}
	Indeed it is clear that the maps in \ref{pointeval} and \ref{degree} are monotone homomorphisms. Conversely, for a given monotone homeomorphism $\phi : \N_{(+)}[X] \to \R_+$ and $n \in \N$, put
	\[
		t_n \coloneqq \phi(1 + X^n) - 1.
	\]
	Then we obtain $\phi(k + \ell \underline{X}) = k + \ell t_n$ for all $k,\ell \in \N$ with $k > 0$ by additivity of $\phi$. Applying this equation results in
	\begin{align*}
		\phi(1 + X^n + X^m + X^{n+m})	& = \frac{1}{3} \left( \phi(1 + 3 X^n) + \phi(1 + 3 X^m) + \phi(1 + 3 X^{n+m}) \right) \\[4pt]
						& = 1 + t_n + t_m + t_{n+m},
	\end{align*}
	as well as
	\begin{align*}
		\phi(1 + X^n + X^m + X^{n+m})	& = \phi(1 + X^n) \phi(1 + X^m) \\[3pt]
						& = (1 + t_n) (1 + t_m),
	\end{align*}
	which together imply that $t_{n+m} = t_n t_m$, and therefore $t_n = t_1^n$ for all $n$. It follows that $\phi(p) = p(t_1)$ for all $p \in \N_{(+)}[X]$, as was to be shown.
	
	For $\psi : \N_{(+)}[X] \to \TR_+$, for any finite $I \subseteq \N$ the value $\psi\left( \sum_{n \in I} p_n X^n \right)$ is independent of the values of the strictly positive coefficients $p_n \in \Nplus$, since $\sum_{n \in I} p_n X^n$ is lower bounded by $\sum_{n \in I} X^n$ and upper bounded by a scalar multiple of it. In other words, $\psi(p)$ only depends on the support $\supp(p)$, which is the set of all $n \in \N$ for which the monomial $X^n$ occurs in $p$. Multiplicativity of $\psi$ shows that with
	\[
		t_n \coloneqq \psi \left( \sum_{j \le n} X^j \right) \in [0,\infty),
	\]
	we have $t_{n + m} = t_n + t_m$, and therefore $t_n = n t_1$. The assumed normalization is $\psi(u) = \psi(2 + X) = 1$, and therefore in fact $t_n = n$. This shows that $\psi$ has the desired form on all polynomials whose support is downward closed. Finally, if $p \in \N_{(+)}[X]$ is arbitrary, then the product $p \sum_{j \le \deg(p)} X^j$ is a polynomial with downward closed support, and applying $\psi$ to it produces $\psi(p) + \deg(p) = 2 \deg(p)$, and therefore implies $\psi(p) = \deg(p)$.

	For the homeomorphism $\Sper{\N_{(+)}[X]} \cong [0,\infty]$, it suffices to show that the bijection is continuous in one direction, since both spaces are compact Hausdorff. It is therefore enough to show that the evaluation maps $p \mapsto p(r)$ converge to the tropical point $\deg$ in the test spectrum topology. This amounts to showing that for every $p \in \N_{(+)}[X]$, we have
	\[
		\deg(p) = \lim_{r \to \infty} \frac{\log p(r)}{\log (2 + r)},
	\]
	which is elementary to verify.
\end{ex}

The situation for polynomials in several variables is more complicated: there are additional monotone homomorphisms $\N_{(+)}[\underline{X}] \to \TR_+$ beyond those arising from optimization over the Newton polytope (\Cref{newtonRplus}), such as the one that sends every polynomial to the highest power of $X_1$ that occurs in it (as a monomial without other variables).
The situation is simpler for the semiring of Laurent polynomials, which we had already touched upon in \Cref{laurent}.

\begin{ex}
	Using $\underline{X} = (X_1,\ldots,X_d)$, consider $\N[\underline{X}, \underline{X}^{-1}]$, the semiring of natural number Laurent polynomials in $d$ variables with respect to the coefficientwise preorder. (All of the following statements apply likewise to the analogously defined $\R_+[\underline{X}, \underline{X}^{-1}]$.) This preordered semiring has a power universal element (\Cref{laurent}).
	We claim that:
	\begin{enumerate}
		\item\label{pointeval2} The monotone homomorphisms $\N[\underline{X}, \underline{X}^{-1}] \to \R_+$ are the evaluation maps $p \mapsto p(e^r)$ indexed by $r \in \R^d$.
		\item\label{degree2} In the additive picture of $\TR_+$, the monotone homomorphisms $\N[\underline{X}, \underline{X}^{-1}] \to \TR_+$ are indexed by $s \in \R^d$, and given by
			\[
				p \longmapsto \max_{\alpha \: \in \: \Newton(p)} \sum_{i=1}^d s_i \alpha_i,
			\]
			where $\Newton(p) \subseteq \R^d$ is the Newton polytope of $p$.
		\item\label{sper2} $\Sper{\N[\underline{X}, \underline{X}^{-1}]}$ is homeomorphic to the closed $d$-dimensional Euclidean ball.
	\end{enumerate}

	Concerning \ref{pointeval2}, we have written the point at which we evaluate as $e^r$ for $r \in \R^d$ in order to make the connection with \ref{degree2} more apparent, which will be relevant for \ref{sper2}.
	To show \ref{pointeval2}, it is enough to note the universal property of $\N[\underline{X}, \underline{X}^{-1}]$ as the semiring freely generated by $d$ invertible variables and it is immediate that these are the evaluation maps at points in $\R^d$ with all coordinates strictly positive. The same argument applies to prove \ref{degree2}. See also \Cref{newtonRplus}.

	For \ref{sper2}, it again suffices to show that the bijection is continuous in one direction. Identifying the ball with the test spectrum as described by \ref{pointeval2} and \ref{degree2} in the obvious way, it is therefore enough to show that every logarithmic evaluation map $\lev_p$ becomes a continuous function on the ball. While continuity at a point in the interior (real part) is obvious, continuity at a point on the boundary (tropical part) follows since the evaluation $p(e^r)$ for $r \to \infty$ is dominated by that monomial which attains the Newton polytope optimization.
\end{ex}

The test spectrum $\Sper{S}$ for given $S$ is not always easy to characterize. This is illustrated also in the following example, which may be of interest in the context of positive polynomials and sums of squares.

\begin{prob}
	\label{sosprob}
	Let $\Sigma_+ \subseteq \R[\underline{X}]$ be the subsemiring of sums of squares plus a strictly positive scalar, preordered such that $p \le q$ if and only if $q - p$ is itself a sum of squares (\Cref{polysos}).
	 Then find a concrete description of all monotone homomorphisms $\Sigma \to \R_+$ and $\Sigma \to \TR_+$.
\end{prob}

Since all the nonzero elements of $\Sigma_+$ do not vanish anywhere, taking the negative degree of vanishing at a point or variety (\Cref{sosvals}) does not give a nontrivial homomorphism $\Sigma \to \TR_+$.

The theory of continuous functions on compact Hausdorff spaces now has some some useful implications.

\begin{lem}
	\label{swapprox}
	Let $C \subseteq \Sper{S}$ be a closed subset and $\phi \in \Sper{S} \setminus C$. Then there are nonzero $x,y \in S$ such that
	\begin{align*}
		\lev_x(\psi) & \le \lev_y(\psi) - 1 \qquad \forall \psi \in C, \\[2pt]
		\lev_x(\phi) & \ge \lev_y(\phi) + 1.
	\end{align*}
\end{lem}

\begin{proof}
	Upon replacing $S$ by $\Frac(S)$ if necessary, and using that $x \mapsto \lev_x$ takes multiplication to addition, we can assume without loss of generality that $S$ is a semifield.
	Indeed under this assumption, we will construct nonzero $z$ such that
	\beq
		\label{zseparate}
		\lev_z(\phi) \ge 1, \qquad \lev_z(\psi) \le -1 \quad \forall \psi \in C,
	\eeq
	which is enough upon writing $z = \frac{x}{y}$ for nonzero $x,y \in S$. Moreover, upon replacing the power universal element $u \in S$ by a power of $u$ if necessary, we can assume $u \ge 2$.
	
	In order to prove that \eqref{zseparate} can be achieved, we will apply the lattice version of the Stone--Weierstra\ss{} theorem~\cite[Theorem~II.7.29]{HS} to the set
	\[
		\mathcal{L} \coloneqq \{ f : \Sper{S} \to \R \text{ continuous} \mid \forall \eps > 0 \: \exists n \in \Nplus, x \in S \setminus \{0\} : \: \| n f - \lev_x \|_\infty < n \eps \}.
	\]
	It is easy to see that $\mathcal{L}$ is closed under addition, under scalar multiplication (using rational approximation, and inverses for negative scalars), contains $1 = \lev_u$ and is closed in the supremum norm $\|\cdot\|_\infty$. Since it trivially contains all $\lev_x$ and these separate points, the lattice version of the Stone--Weierstra\ss{} theorem shows that it is enough to prove that $\mathcal{L}$ is closed under max in order to conclude that $\mathcal{L} = C(\Sper{S})$. 

	So let $f,g \in \mathcal{L}$ and consider $\max(f,g) : \Sper{S} \to \R$. By assumption we have nonzero $x,y \in S$ and $n \in \N$ such that
	\[
		\| n f - \lev_x \|_\infty < n \eps, \qquad \| n g - \lev_y \|_\infty < n \eps,
	\]
	where we can assume that $n$ is the same for both by using the smallest common multiple if necessary (and replacing the relevant elements of $S$ by the corresponding powers). By the Maslov dequantization of \Cref{maslov} and boundedness of the logarithmic evaluation functions, we have
	\[
		\left\| \max\left( \frac{\lev_x}{\lev_2}, \frac{\lev_y}{\lev_2} \right) - k^{-1} \log_2 \left( 2^{k \frac{\lev_x}{\lev_2}} + 2^{k \frac{\lev_y}{\lev_2}} \right) \right\|_\infty < \eps
	\]
	for all natural $k \gg 1$ on the entire real part of the test spectrum (characterized by $\lev_2 \neq 0$). Therefore by \Cref{philevadd}, on the real part we have
	\[
		\left\| \max\left( \frac{\lev_x}{\lev_2}, \frac{\lev_y}{\lev_2} \right) - k^{-1} \frac{\lev_{x^k + y^k}}{\lev_2} \right\|_\infty < \eps.
	\]
	for all $k \gg 1$. Since $\lev_2 \le \gamma$ for some constant $\gamma > 0$, we can also write this as
	\[
		\left\| k \max \left( \lev_x, \lev_y \right) - \lev_{x^k + y^k} \right\|_\infty < k \eps \gamma
	\]
	for all $k \gg 1$, which now holds on all of $\Sper{S}$ since the left-hand side even vanishes on the tropical part. We fix any large enough $k$ for which this holds.
	With $z \coloneqq x^k + y^k$ and the above $n$, we therefore get upon combining this with the previous inequalities,
	\[
		\left\| nk \max(f,g) - \lev_{z^n} \right\| < (2 + \gamma) nk \eps,
	\]
	which is enough to conclude $\max(f,g) \in \mathcal{L}$.

	In conclusion, by $\mathcal{L} = C(\Sper{S})$, we now know that for every continuous $f : \Sper{S} \to \R$ and every $\eps > 0$ there are $n \in \Nplus$ and nonzero $x \in S$ such that
	\[
		|n f(\psi) - \lev_x(\psi)| < n \eps
	\]
	for all $\psi \in \Sper{S}$. Now since every compact Hausdorff space is completely regular, there is $f : \Sper{S} \to [-2,+2]$ such that $f(\psi) = -2$ for all $\psi \in C$ and $f(\phi) = +2$. Taking $\eps \coloneqq 1$ in the approximation statement implies that we have $n \in \Nplus$ and nonzero $z \in S$ with
	\[
		\lev_z(\phi) \ge +n, \qquad \lev_z(\psi) \le -n \quad \forall \psi \in C,
	\]
	This in particular gives the desired \eqref{zseparate} by using $n \ge 1$.
\end{proof}

\subsection*{The Vergleichsstellensatz}

We can now formulate and prove the desired result for preordered semirings.
Later in this section, we will rederive the classical Positivstellensatz of Krivine--Kadison--Dubois from it (\Cref{kkd}). Hence a good way to think about this result is that it is a non-Archimedean generalization of Krivine--Kadison--Dubois.

\begin{thm}
	\label{simpler1}
	Let $S$ be a preordered semiring with $1 \ge 0$ and a power universal element $u$, and let $x, y \in S$ nonzero.
	Then the following are equivalent:
	\begin{enumerate}
		\item\label{newpos} $\phi(x) \le \phi(y)$ for all monotone homomorphisms $\phi : S \to \R_+$, and all $\phi : S \to \TR_+$ with $\phi(u) > 1$.
		\item\label{newpower} For every $\eps > 0$, we have
			\[
				x^n \le u^{\lfloor \eps n \rfloor} y^n
			\]
			for all $n \gg 1$.
	\end{enumerate}
	Moreover, suppose that $\phi(x) < \phi(y)$ for all such $\phi$. Then also the following hold:
	\begin{enumerate}[resume]
		\item\label{asymp1} There is $k \in \N$ such that
			\[
				u^k x^n \le u^k y^n \qquad \forall n \gg 1.
			\]
		\item\label{asymp2} If $y$ is power universal as well, then
			\[
				x^n \le y^n \qquad \forall n \gg 1.
			\]
		\item\label{catal} There is nonzero $a \in S$ such that
			\[
				a x \le a y.
			\]
			Moreover, there is $k \in \N$ such that $a \coloneqq u^k \sum_{j=0}^n x^j y^{n-j}$ does the job for any $n \gg 1$.
	\end{enumerate}
\end{thm}

\begin{proof}
	We assume $1 > 0$ without loss of generality.
	By \Cref{nonzeros}, we know that such $S$ is a zerosumfree semidomain. Hence $\Frac(S)$ exists and is a preordered semifield with $1 > 0$ and a power universal element, which we also denote by $u$.

	The implication from \ref{newpower} to \ref{newpos} is straightforward based on the fact that every monotone homomorphism $\phi : S \to \R_+$ or $\phi : S \to \TR_+$ has trivial kernel.
	For the remaining claims, we first prove the following weaker version:
	\begin{enumerate}
		\setcounter{enumi}{5}	
		\item\label{newaux} If $\phi(x) < \phi(y)$ for all $\phi \in \Sper{S}$, then there is nonzero $a \in S$ with $a x \le a y$.
	\end{enumerate}
	By \Cref{frac_order_embed}, it is enough to prove that $x \le y$ holds in $\Frac(S)$. By \Cref{abstract_semifield_pss}, it is moreover enough to consider the case that $\Frac(S)$ is totally preordered (as a semifield). Note that $u$ is still power universal in any such total extension; this is where the assumption $1 > 0$ enters via the universality criterion of \Cref{univ_old}.
	But then \Cref{arch_truncate} produces a monotone homomorphism $\phi : \Frac(S) \to \mathbb{K}$ with $\mathbb{K} \in \{\R_+, \TR_+\}$ such that $\phi(u) > 1$.
	In the tropical case, $\phi(u) > 1$ and replacing $\phi$ by a suitable power of itself achieves the normalization $\phi(u) = e$, and hence $\phi(u) \in \Sper{S}$.
	Now if we had $x > y$, then we would obtain $\phi(x) \ge \phi(y)$, contradicting the assumed $\phi(x) < \phi(y)$.
	This establishes \ref{newaux}.

	Now assuming \ref{newpos} and fixing $\eps > 0$ as in \ref{newpower}, we choose a positive rational $\frac{\ell}{m} < \frac{\eps}{3}$ and consider $\tilde{x} \coloneqq x^m$ and $\tilde{y} \coloneqq u^\ell y^m$. Then $\phi(\tilde{x}) < \phi(\tilde{y})$ for all $\phi \in \Sper{S}$, and therefore \ref{newaux} gives us nonzero $a \in S$ such that $a \tilde{x} \le a \tilde{y}$.
	Upon chaining inequalities, this implies
	\[
		a \tilde{x}^j \le a \tilde{y}^j
	\]
	for all $j \in \Nplus$. Choosing $k \in \N$ such that $a \le u^k$ and $1 \le a u^k$, we obtain upon also plugging in the definitions of $\tilde{x}$ and $\tilde{y}$,
	\[
		x^{mj} \le a u^k x^{mj} \le a u^{k + \ell j} y^{mj} \le u^{2 k + \ell j} y^{mj}.
	\]
	Since $\frac{\ell}{m} < \frac{\eps}{3}$, this results in
	\[
		x^{mj} \le u^{\lfloor \frac{\eps m j}{2} \rfloor} y^{mj}
	\]
	for all $j \gg 1$, which is an inequality of the desired form. Choosing another $k$ such that $x \le y u^k$, we obtain further
	\[
		x^{mj+1} \le u^{\lfloor \frac{\eps m j}{2} \rfloor + k} y^{mj+1} \le u^{\lfloor \eps m j \rfloor} y^{mj+1},
	\]
	where the second inequality is for sufficiently large $j$. Fixing such $j$ and weakening a bit further, we therefore have
	\[
		x^{mj} \le u^{\lfloor \eps m j \rfloor} y^{mj}, \qquad x^{mj+1} \le u^{\lfloor \eps (m j + 1) \rfloor} y^{mj+1}.
	\]
	By multiplying powers of these two inequalities, and using that every sufficiently large natural number is a sum of positive integer multiples of $mj$ and $mj+1$, we therefore obtain $x^n \le u^{\lfloor \eps n \rfloor} y^n$ for all $n \gg 1$, as was to be shown.

	We now assume $\phi(x) < \phi(y)$ for all $\phi \in \Sper{S}$ and prove the remaining items \ref{asymp1}--\ref{catal}. 
	For \ref{asymp1}, we use that $\lev_x, \lev_y : \Sper{S} \to \R_+$ are continuous real-valued functions on a compact Hausdorff space by \Cref{chaus}. Since the function $\phi \longmapsto \lev_y(\phi) - \lev_x(\phi)$ is strictly positive by assumption, the compactness implies that it is even bounded away from zero by some $\eps > 0$. Plugging in the definition of $\lev_x$ and $\lev_y$ lets us write this in the form
	\[
		\frac{\phi(x)}{\phi(y)} < \phi(u)^{-\eps} \qquad \forall \phi \in \Sper{S}.
	\]
	Upon choosing positive rational $\frac{\ell}{m} < \eps$, we therefore obtain
	\[
		\phi(x) < \phi(u)^{-\frac{\ell}{m}} \phi(y)
	\]
	for all $\phi$, or equivalently $\phi(\tilde{x}) < \phi(\tilde{y})$ with $\tilde{x} \coloneqq u^\ell x^m$ and $\tilde{y} \coloneqq y^m$. Applying \ref{newpower} with $\eps \coloneqq \ell$, we can find $k,n \in \Nplus$ with $k < \ell n$ and such that $\tilde{x}^n \le u^k \tilde{y}^n$. Equivalently,
	\[
		u^{\ell n} x^{mn} \le u^k y^{mn}.
	\]
	To summarize and reinitialize symbols, we therefore can find large enough $k$ and $m$ such that
	\beq
		\label{asymp_strong}
		u^k x^m \le u^{k-1} y^m.
	\eeq
	This can be weakened to $u^k x^m \le u^k y^m$, showing that the desired inequality holds for \emph{some} exponent. Continuing, let $n$ be such that $x \le y u^n$. Then we obviously have $u^{kn} x^{mn} \le u^{kn} y^{mn}$ by the previous, but also
	\[
		u^{kn} x^{mn + 1} \le u^{(k+1)n} x^{mn} y \le u^{kn} y^{mn + 1},
	\]
	which differs from $u^{kn} x^{mn} \le u^{kn} y^{mn}$ only in having one larger exponent of $x$ and $y$, respectively. Now since every sufficiently large natural number is a sum of positive integer multiples of $mn$ and $mn + 1$, we finally obtain the claim, since the set of all $j \in \N$ for which 
	\[
		u^{kn} x^j \le u^{kn} y^j
	\]
	holds is closed under addition.

	To prove \Cref{asymp2}, we choose $u \coloneqq y$ and use \eqref{asymp_strong} in the form
	\[
		y^k x^m \le y^k y^{m - 1}
	\]
	for suitably large $k$ and $m$. To bring this into a simpler form, note that $k$ can always be increased by multiplying by a power of $y$, while $m$ can be replaced by a multiple of it by chaining and using $y \ge 1$. We can therefore assume $m = k$ for simplicity, giving
	\[
		y^k x^k \le y^{2k - 1},
	\]
	which we know to hold for infinitely many $k$. Using chaining by induction on $j$, we get
	\[
		y^k x^{jk} \le y^{2k + (j-1)(k-1)}
	\]
	for all $j\in\Nplus$. Since $y \ge 1$, we can weaken this to
	\[
		x^{jk} \le y^{2k + (j-1)(k-1)} \le y^{jk},
	\]
	where the second inequality holds for $j \gg 1$. This proves the desired inequality for some exponent. Choosing $j$ large enough so that the difference of exponents $\ell \coloneqq jk - [2k + (j-1)(k-1)]$ is large enough for $x \le y^\ell$ to be the case, we moreover obtain
	\[
		x^{jk + 1} \le x^{jk} y^\ell \le y^{2k + (j-1)(k-1) + \ell} \le y^{jk + 1}.
	\]
	The claim that the desired inequality holds for all sufficiently large exponents now follows by the same reasoning as for \ref{asymp1} in the previous paragraph.

	Finally, we derive the particular form of $a$ in \ref{catal} from \ref{asymp1}. The following argument is essentially~\cite[Lemma~5.4]{ocm} and generalizes \Cref{geom_fourier}. The assumed inequality in the form $u^k x^{n+1} \le u^k y^{n+1}$ implies
	\begin{align*}
		\left( u^k \sum_{\ell=0}^n x^\ell y^{n-\ell} \right) x	& = u^k \left( \sum_{\ell=0}^{n-1} x^{\ell+1} y^{n-\ell} + x^{n+1} \right) \\
									& \le u^k \left( y^{n+1} + \sum_{\ell=1}^n x^\ell y^{n+1-\ell} \right) \\
									& = \left( u^k \sum_{\ell=0}^n x^\ell y^{n-\ell} \right) y,
	\end{align*}
	as was to be shown.
\end{proof}

We record the application of this result to the semiring of Laurent polynomials.
Other applications, such as the one to representation theory sketched in the introduction, are going to be considered elsewhere~\cite{rep_app}.

\begin{ex}
	\label{ex_laurent}
	Suppose that $p, q \in \R_+[\underline{X}, \underline{X}^{-1}]$ are nonzero Laurent polynomials such that:
	\begin{itemize}
		\item $p(r) < q(r)$ for all $r \in \R_{> 0}^d$.
		\item The Newton polytope of $p$ is contained in the interior of the Newton polytope of $q$. 
	\end{itemize}
	Then, using $u \coloneqq 1 + \sum_i (X_i + X_i^{-1})$ as in \Cref{laurent}, we have:
	\begin{itemize}
		\item There is another nonzero Laurent polynomial $a \in \R_+[\underline{X}, \underline{X}^{-1}]$ such that
			\[
				a p \le a q
			\]
			holds coefficientwise. Moreover, there is $k \in \N$ such that $a \coloneqq u^k \sum_{j=0}^n p^j q^{n-j}$ does the job for every $n \gg 1$.
		\item There is $k \in \N$ such that
			\[
				u^k p^n \le u^k q^n
			\]
			holds coefficientwise for all $n \gg 1$.
	\end{itemize}
	It is worth noting that for Laurent polynomials and the coefficientwise preorder in particular, necessary and sufficient conditions for the existence of nonzero $a$ with $a p \le a q$ have been derived by Handelman, who uses the traditional formulation in terms of ordered rings~\cite[Theorem~V.6.C]{handelman}. Of course, our result applies much more generally, and it is not surprising that more specific information can be obtained in particular cases like this one. (Which is not to say that doing so is easy.)
\end{ex}

\subsection{An extension theorem for monotone homomorphisms}

We use \Cref{simpler1} in order to prove a result on extending monotone homomorphisms to $\R_+$ and $\TR_+$, taking the form of an injectivity property for these two preordered semifields.

\begin{thm}
	\label{extensionthm}
	Let $S$ be a preordered semiring with $1 \ge 0$ and a power universal element $u$, and let $T \subseteq S$ be a subsemiring with $u \in T$ and carrying the induced preorder.
	Then every monotone homomorphism of the form
	\[
		\phi : T \to \R_+, \qquad \textrm{or} \qquad \phi : T \to \TR_+
	\]
	can be extended to a monotone homomorphism on $S$ (with the same codomain).
\end{thm}

\begin{proof}
	This amounts to proving that the continuous map $\Sper{S} \to \Sper{T}$ is surjective. Since $\Sper{S}$ is compact, its image in $\Sper{T}$ is a closed subspace. Suppose that there is $\psi \in \Sper{T}$ not in the image. We can then apply \Cref{swapprox} to obtain nonzero $x,y \in T$ such that
	\begin{align*}
		\lev_x(\phi) & \le \lev_y(\phi) - 1 \qquad \forall \phi \in \Sper{S}, \\[2pt]
		\lev_x(\psi) & \ge \lev_y(\psi) + 1.
	\end{align*}
	Applying \Cref{simpler1} together with the first inequality shows in particular that there is nonzero $a \in S$ with $a x \le a y$; the particular form of $a$ given in \ref{catal} shows that we can even assume $a \in T$. But then by monotonicity and preservation of multiplication by the logarithmic evaluation maps, we also get
	\[
		\lev_x(\psi) \le \lev_y(\psi),
	\]
	in contradiction with the above.
\end{proof}

We now present a first theoretical application of this extension theorem towards a better understanding of test spectra. With $u \in S$ a power universal element, we write $\N[u]$ for the polynomial semiring $\N[X]$ equipped with the semiring preorder in which
\[
	p \le q \qquad \Longleftrightarrow \qquad p(u) \le q(u).
\]
\Cref{univ_old} shows that $u$ is still a power universal element in $\N[u]$. Moreover, $\Sper{\N[u]}$ is a closed subset of $[0,\infty]$, where $\infty$ corresponds to the (not necessarily monotone) degree homomorphism $\N[u] \to \TR_+$.

We now present three lemmas which concern the interaction between the real part and the tropical part of $\Sper{S}$, and on how this interaction relates to the structure of $\N[u]$.
The first lemma is concerned with when the real part and the tropical part of $\Sper{S}$ are disconnected.

\begin{lem}
	\label{sper_disconnect}
	The following are equivalent:
	\begin{enumerate}
		\item\label{disconnect} The real and tropical part of $\Sper{S}$ are disconnected.
		\item\label{realclosed} The real part of $\Sper{S}$ is closed.
		\item\label{uboundvalue} The values $\phi(u)$ are bounded as $\phi : S \to \R_+$ varies.
		\item\label{realcompact} The set of monotone homomorphisms $\phi : S \to \R_+$ is compact in the weak topology induced by the evaluation maps $x \mapsto \phi(x)$.
		\item\label{uboundsquare} There is $n \in \N$ such that for every $\ell \in \N$ and $\phi \in \Sper{S}$,
			\[
				\phi(\ell u) \le \phi(\ell n + u^2).
			\]
		\item\label{Nualternative} The degree homomorphism $\N[u] \to \TR_+$ is either not monotone or an isolated point of $\Sper{\N[u]}$.
	\end{enumerate}
\end{lem}

\begin{proof}
	Since the tropical part of $\Sper{S}$ is closed, the equivalence of \ref{disconnect} and \ref{realclosed} is by definition.

	Condition \ref{realclosed} is equivalent to the statement that the set of all $\phi \in \Sper{S}$ which satisfy $\lev_2(\phi) \neq 0$ is closed. Since closedness implies compactness, \ref{realclosed} implies that there is $\eps > 0$ such that $\lev_2(\phi) > \eps$ for all $\phi : S \to \R_+$. Using this together with the definition $\lev_2(\phi) = \frac{\log 2}{\log \phi(u)}$ produces a bound of the desired form.

	From \ref{uboundvalue} to \ref{realcompact}, it is enough to show that the values $\phi(x)$ are bounded as $\phi : S \to \R_+$ varies for every fixed $x \in S$. Since $x \le u^k$ for sufficiently large $k \in \N$, this follows from the assumption together with the properties of monotone homomorphisms. The converse from \ref{realcompact} to \ref{uboundvalue} is clear.

	From \ref{realcompact} to \ref{uboundsquare}, we can thus assume \ref{uboundvalue}. But then there already is $n \in \N$ such that $\phi(u) \le \phi(n)$ for all $\phi : S \to \R_+$, and we anyway have $\phi(\ell u) = \phi(u) \le \phi(u^2)$ for $\phi : S \to \TR_+$. Thus the desired inequality follows in both cases.

	By \Cref{polyspectrum} and $u \ge 1$, the test spectrum $\Sper{\N[u]}$ is a closed subset of $[1,\infty]$. Thus upon assuming \ref{uboundsquare}, the only way in which \ref{Nualternative} can fail is if the evaluation map
	\[
		\N[u] \longrightarrow \R_+, \qquad p \longmapsto p(r)
	\]
	is monotone for all sufficiently large $r \in (0,\infty)$. But then choosing any such $r > n$, we get
	\[
		\ell r > \ell n + r^2
	\]
	for sufficiently large $\ell$, contradicting the assumption.

	Finally assuming \ref{Nualternative}, the disconnectedness property of \ref{disconnect} clearly holds for $\N[u]$. It then follows for $S$ itself from \Cref{extensionthm}.
\end{proof}

The second lemma is concerned with when $\Sper{S}$ does not have a tropical part at all.

\begin{lem}
	\label{notroplem}
	The following are equivalent:
	\begin{enumerate}
		\item\label{notropical} There is no monotone homomorphism $S \to \TR_+$ other than the trivial one given by $S \to \B \to \TR_+$.
		\item\label{ubound} There is $n \in \N$ such that $\phi(u) \le \phi(n)$ for all $\phi \in \Sper{S}$.
		\item\label{nodegree} The degree homomorphism $\N[u] \to \TR_+$ is not monotone.
		\item\label{decreasedegree} There is $p \in \N[u]$ such that
			\[
				u^{\deg(p) + 1} \le p(u)
			\]
			in $S$.
	\end{enumerate}
\end{lem}

\begin{proof}
	From \ref{notropical} to \ref{ubound}, we know that the values $\phi(u)$ are bounded by \Cref{sper_disconnect}\ref{uboundvalue}, which is enough by $\phi(n) = n$. The converse is clear since $\phi(n) = \phi(1) < \phi(u)$ for tropical $\phi$.

	Assuming \ref{notropical}, we obtain \ref{nodegree} by \Cref{extensionthm}: if the degree homomorphism was monotone, then it would have to extend to a tropical point $\phi \in \Sper{S}$. The converse is clear by restriction to $\N[u]$ and \Cref{polyspectrum}.
	
	Finally, \ref{decreasedegree} clearly implies \ref{nodegree}. Conversely if the degree homomorphism $\N[u] \to \TR_+$ is not monotone, then there must be $p, q \in \N[X]$ with $q(u) \le p(u)$ and $\deg(q) > \deg(p)$. But since $u^{\deg(q)} \le q(u)$, because in particular the leading coefficient is in $\Nplus$, property~\ref{decreasedegree} now follows as $u^{\deg(p) + 1} \le u^{\deg(q)} \le q(u) \le p(u)$.
\end{proof}

The third lemma considers the question of when the real part of $\Sper{S}$ is dense. For any $x, y \in S$, we consider the preordered semiring $\N[u,x,y]$ defined in the same way as $\N[u]$.

\begin{lem}
	\label{denselem}
	The following are equivalent:
	\begin{enumerate}
		\item\label{realdenseS} The real part of $\Sper{S}$ is dense.
		\item\label{realdenseNux} For every $x, y \in S$, the real part of $\Sper{\N[u,x,y]}$ is dense.
	\end{enumerate}
\end{lem}

\begin{proof}
	Assuming \ref{realdenseS}, condition \ref{realdenseNux} follows by extending any $\psi : \N[u,x,y] \to \TR_+$ to $S \to \TR_+$ via \Cref{extensionthm}, approximating that extension by a real point, and restricting the latter back to $\N[u,x,y]$.

	Conversely, we assume \ref{realdenseNux} and derive \ref{realdenseS} using proof by contradiction. If $\psi : S \to \TR_+$ is not in the closure of the real part, then by \Cref{swapprox} we have nonzero $x,y \in S$ such that
	\begin{align*}
		\lev_x(\phi) & \le \lev_y(\phi) - 1 \qquad \forall \phi : S \to \R_+, \\[2pt]
		\lev_x(\psi) & \ge \lev_y(\psi) + 1.
	\end{align*}
	But then again since every monotone homomorphism $\N[u,x,y] \to \R_+$ can be extended to $S$, it follows that the same separation holds with respect to $\N[u,x,y]$ as well, implying that the restriction of $\psi$ to $\N[u,x,y]$ is not contained in the closure of the real part of $\Sper{\N[u,x,y]}$.
\end{proof}

\begin{ex}
	\label{polydense}
	By \Cref{laurent}, the real part of $\Sper{\N[\underline{X}, \underline{X}^{-1}]}$ is dense, and similarly for $\R_+[\underline{X}, \underline{X}^{-1}]$.
\end{ex}

\section{Rederiving existing Vergleichsstellens\"atze and Positivstellens\"atze}

\subsection*{Strassen's Vergleichsstellensatz and its generalizations}

We now show how Strassen's Vergleichsstellensatz~\cite[Corollary~2.6]{strassen}, in the generalized form given by Zuiddam~\cite[Theorem~2.2]{zuiddam}, is a consequence of \Cref{simpler1}.
We do so by restricting to the case where $u \coloneqq 2$ is a power universal element.

\begin{thm}
	\label{sps}
	Let $S$ be a preordered semiring with
	$1 \ge 0$, and such that for every nonzero $x \in S$ there is $\ell \in \Nplus$ such that
	\[
		x \le \ell, \qquad 1 \le \ell x.
	\]
	Then for nonzero $x, y \in S$, the following are equivalent:
	\begin{enumerate}
		\item\label{spos} $\phi(x) \le \phi(y)$ for all monotone homomorphisms $\phi : S \to \R_+$.
		\item\label{spower} For every $\eps > 0$, we have
			\[
				x^n \le 2^{\lfloor \eps n \rfloor} y^n
			\]
			for all $n \gg 1$.
	\end{enumerate}
	Moreover, suppose that $\phi(x) < \phi(y)$ for all such $\phi$. Then also the following hold:
	\begin{enumerate}[resume]
		\item\label{sasymp1} There is $k \in \N$ such that
			\[
				2^k x^n \le 2^k y^n \qquad \forall n \gg 1.
			\]
		\item\label{sasymp2} If $y \ge 1$ is such that $y^\ell \ge 2$ for some $\ell$, then also
			\[
				x^n \le y^n \qquad \forall n \gg 1.
			\]
		\item\label{scatal} There is nonzero $a \in S$ such that
			\[
				a x \le a y.
			\]
			Moreover, there is $k \in \N$ such that $a \coloneqq 2^k \sum_{j=0}^n x^j y^{n-j}$ does the job for any $n \gg 1$.
	\end{enumerate}
\end{thm}

Modulo inconsequentially weaker assumptions\footnote{Strassen and Zuiddam make the slightly stronger assumption that the homomorphism $\N \hookrightarrow S$ must be an order embedding rather than merely $1 \ge 0$. This is not necessary, since the statement trivially holds without it: if $m + 1 \le m$ in $S$ for some $m \in \N$, then $\Sper{S} = \emptyset$, so that the assumption of the theorem is satisfied for any nonzero $x, y \in S$. And indeed the conclusion always holds too: we have $x \le \ell y$ for sufficiently large $\ell \in \N$, which implies $x^n \le m x^n$ for all $n$ because of $\ell^n \le m$.}, the equivalence of \ref{spos} and \ref{spower} is exactly Zuiddam's version of Strassen's Vergleichsstellensatz. In the next and final subsection below, we will show how this result can also be used to rederive the classical Positivstellensatz of Krivine--Kadison--Dubois.

As another indication of the strength of Strassen's Vergleichsstellensatz, it is worth noting that our additional statements \ref{sasymp1}--\ref{scatal} follow directly from the equivalence of \ref{spos} and \ref{spower} without any further algebraic considerations. It is enough to specialize the relevant parts of our proof of \Cref{simpler1} to $u = 2$.

We now turn to our generalization of Strassen's Vergleichsstellensatz developed in~\cite{our_spss}, stating it in a slightly simpler but manifestly equivalent form which allows for scalar multiplication by positive rationals\footnote{In order to recover the original version, replace $S$ by $S \otimes \Q_+$, as in the proof of \Cref{quasicompl_pss}.}, and then reproving it from \Cref{simpler1}.

\begin{thm}[{\cite[Theorem~2.12]{our_spss}}]
	\label{spss}
	Let $S$ be a preordered semiring with $\Q_+ \subseteq S$ and $1 \ge 0$ and having a power universal element $u \in S$. Then for nonzero $x,y \in S$, the following are equivalent:
	\begin{enumerate}
		\item\label{oldpos} $\phi(x) \le \phi(y)$ for all monotone homomorphisms $\phi : S \to \R_+$.
		\item\label{oldadd} For every $r \in \R_+$ and $\eps > 0$, there exist a polynomial $p \in \Q_+[X]$ and nonzero $a \in S$ such that $p(r) \le \eps$ and
			\[
				a x \le a(y + p(u)).
			\]
		\item\label{oldcat} For every $r \in \R_+$ and $\eps > 0$, there exist a polynomial $p \in \Q_+[X]$ and nonzero $a \in S$ such that $p(r) \le 1 + \eps$ and
			\[
				a x \le p(u)\, a y.
			\]
		\item\label{oldpower} For every $r \in \R_+$ and $\eps > 0$, there exist a polynomial $p \in \Q_+[X]$ and $n \in \Nplus$ such that $p(r) \le (1 + \eps)^n$ and
			\[
				x^n \le p(u)\, y^n.
			\]
	\end{enumerate}
\end{thm}

The polynomials $p$ that appear in the following proof have a very specific form which is considerably simpler than the one that can be extracted from our earlier proof in~\cite{our_spss}.

\begin{proof}
	It is again straightforward to derive \ref{oldpos} from either of the other assumptions. So we assume \ref{oldpos}, and also assume without loss of generality that the $r$ and $\eps$ appearing in each goal are positive rational. We also fix $k \in \Nplus$ with $x \le u^{k-1}$ and $y u^k \ge 1$, so that in particular $x \le y u^{2k}$.

	We then obtain \ref{oldadd} upon applying \Cref{simpler1} to $\tilde{x} \coloneqq x$ and $\tilde{y} \coloneqq y + p(u)$ with $p \coloneqq \eps r^{-k} X^k$, since for any tropical $\phi$ we have
	\[
		\phi(\tilde{x}) < \phi(u) \phi(u^{k-1}) \le \phi(y) + \phi(u^k) = \phi(y + \eps r^{-k} u^k) = \phi(\tilde{y}),
	\]
	and clearly $\phi(\tilde{x}) < \phi(\tilde{y})$ by the assumption~\ref{oldadd}. Hence \Cref{simpler1} implies the desired $a \tilde{x} \le a \tilde{y}$.
	For \ref{oldcat}, we similarly use $\tilde{y} \coloneqq p(u) y$ with $p \coloneqq 1 + \eps r^{-2k} X^{2k}$.
	For \ref{oldpower}, we use the same $\tilde{y}$ and $\tilde{x}$ and apply \Cref{simpler1}\ref{newpower}, which implies \ref{oldpower} with $2\eps$ in place of $\eps$.
\end{proof}

There is also a generalization of Strassen's Vergleichsstellensatz due to Vrana~\cite[Theorem~2]{vrana}, but we have so far not been able to obtain it directly\footnote{i.e.~without proceeding through the route of Vrana's own proof.} from \Cref{simpler1}.

\subsection*{A Vergleichsstellensatz with quasi-complements}

We end the paper with a few observations on the implications of Strassen's Vergleichsstellensatz in the form of \Cref{sps}, noting that it implies the classical Positivstellensatz of Krivine--Kadison--Dubois.
Recall from \Cref{quasicompl_defn} that a semiring $S$ has quasi-complements if for every $a \in S$ there are $b \in S$ and $n \in \N$ with $a + b = n$. 

\begin{thm}
	\label{quasicompl_pss}
	Let $S$ be a preordered semiring with quasi-complements and $1 \ge 0$, and $x, y \in S$ nonzero.
	Then the following are equivalent:
	\begin{enumerate}
		\item\label{qc_geomnonstrict} $\phi(x) \le \phi(y)$ for all monotone homomorphisms $\phi : S \to \R_+$.
		\item\label{qc_algnonstrict} For every $n \in \Nplus$ there is $\ell \in \Nplus$ such that
			\[
				\ell n x \le \ell ( n y + 1 ).
			\]
	\end{enumerate}
	Moreover, suppose that $\phi(x) < \phi(y)$ for all such $\phi$. Then also:
	\begin{enumerate}[resume]
		\item\label{qc_algstrict} There is $\ell \in \Nplus$ such that
				\beq
					\label{qc_ineq}
					\ell (x + 1) \le \ell (y + 1).
				\eeq
	\end{enumerate}
\end{thm}

By chaining inequalities, \eqref{qc_ineq} automatically implies $\ell (n x + 1) \le \ell (n y + 1)$ for all $n \in \N$.

\begin{proof}
	The implication from \ref{qc_algnonstrict} to \ref{qc_geomnonstrict} is straightforward as usual, so we focus on the other claims.

	Assuming $1 > 0$ without loss of generality, it is convenient to replace $S$ by $S \otimes \Q_+$, by which we mean the preordered semiring whose elements are equivalence classes of formal fractions $\frac{x}{m}$ with $m \in \Nplus$ and $x \in S$, where the equivalence relation is $\frac{x}{m} \sim \frac{y}{n}$ if and only if there is $\ell \in \Nplus$ with $\ell n x = \ell m y$ in $S$. This becomes a semiring with respect to the obvious addition and multiplication induced from $S$, and a preordered semiring with respect to
	\[
		\frac{x}{m} \le \frac{y}{n} \quad :\Longleftrightarrow \quad \exists \ell \in \Nplus : \: \ell n x \le \ell m y.
	\]
	It is then straightforward to see that $S$ has the desired properties if and only if $S \otimes \Q_+$ does, since in particular $S \otimes \Q_+$ still has quasi-complements.
	We therefore can assume without loss of generality that we have a homomorphism $\Q_+ \to S$, or in other words that scalar multiplication by arbitrary positive rationals is available.

	Moreover, we can replace the given preorder on $S$ by
	\[
		a \preceq b \quad :\Longleftrightarrow \quad a + 1 \le b + 1.
	\]
	This is easily seen to be a semiring preorder again, where the monotonicity of multiplication by some other element $c \in S$ relies on the existence of a quasi-complement for $c$.
	Since the claim holds for $(S,\le)$ if and only if it holds for $(S,\preceq)$ and the latter semiring is order cancellative (by quasi-complements again), we can assume without loss of generality that $S$ is order cancellative.

	Under these two extra assumption, we now prove that $\phi(x) < \phi(y)$ for all $\phi$ implies that $x \le y + \eps$ for all $\eps \in \Qplus$, which we will then sharpen at the end.
	Consider the subsemiring
	\[
		S_+ \coloneqq \{ r + a  \mid r \in \Qplus, \, a \in S \} \cup \{0\},
	\]
	equipped with the induced preorder. Then $S_+$ is a preordered semiring with $1 > 0$ and the power universal element $u \coloneqq 2$, since $a + b = n$ shows in particular that $a \le n$, and every nonzero $a \in S_+$ is already lower bounded by a positive scalar by construction. Upon adding $1$ to the given elements $x$ and $y$ if necessary, we can assume $x,y \in S_+$.

	Thus Strassen's Vergleichsstellensatz in the form of \Cref{spss} applies, and $\phi(x) < \phi(y)$ for all $\phi$ implies that we have nonzero $a \in S_+$ with $a x \le a y$. By choosing a quasi-complement and rescaling, we have $a + b = 1$ for suitable $b \in S_+$. Suppose $k \in \Nplus$ is such that $a,b \ge k^{-1}$. Then by $a + b = 1$ and order cancellativity, we also get $a,b \le \frac{k-1}{k}$. The standard telescoping argument for the geometric series shows that for every $m \in \N$,
	\[
		a \sum_{i=0}^m b^i + b^{m+1} = 1.
	\]
	Thus with $c \coloneqq a \sum_{i=0}^m b^i$, we get $c x \le c y$ and
	\[
		\left( 1 - \left(\frac{k-1}{k}\right)^{m+1} \right) \le c \le 1.
	\]
	Therefore also
	\[
		\left( 1 - \left(\frac{k-1}{k}\right)^{m+1} \right) x \le c x \le c y \le y
	\]
	for every $m$. Since the fraction is $< 1$, taking $m \gg 1$ proves that $(1 - \eps) x \le y$ for every $\eps \in \Qplus$, or equivalently $x \le (1 + \eps) y$. Since $y$ is upper bounded by a constant, this implies the desired $x \le y + \eps$ after a suitable adjustment to $\eps$.
	
	Thus assuming merely $\phi(x) \le \phi(y)$ for all $\phi$, we can apply the above statement with $y + \eps$ in place of $y$ and conclude $x \le y + 2 \eps$, which proves \ref{qc_algnonstrict}.

	Finally if $\phi(x) < \phi(y)$ for all $\phi$, then we can appeal to compactness of $\Sper{S}$ again in order to find $\eps \in \Qplus$ satisfying also $\phi(x + \eps) < \phi(y)$, so that we get $x + \eps \le y + \eps$. The assumed order cancellativity then produces the desired $x \le y$.
\end{proof}

This implies in particular the following version of the Positivstellensatz of Krivine--Kadison--Dubois in the form given by Becker and Schwartz~\cite{BS}.

\begin{cor}
	\label{kkd}
	Let $R$ be a ring together with a subsemiring $R_+ \subseteq R$ such that
	\[
		R = \N - R_+,
	\]
	and a subset $M \subseteq R$ such that
	\[
		1 \in M, \qquad M + M \subseteq M, \qquad R_+ M \subseteq M.
	\]
	Then the following are equivalent for every $x \in R$:
	\begin{enumerate}
		\item $\phi(x) \ge 0$ for all homomorphisms $\phi : R \to \R$ satisfying $\phi(M) \subseteq \R_+$.
		\item For all $n \in \N$ there is $\ell \in \Nplus$ such that $\ell(1 + n x) \in M$.
	\end{enumerate}
\end{cor}

The assumption $R = \N - R_+$ means exactly that the positive cone $R_+ \subseteq R$ (sometimes also called a \emph{preprime}) is Archimedean in the usual sense of real algebra.

\begin{proof}
	This follows from an application of \Cref{quasicompl_pss} with $S \coloneqq R_+$, equipped with the semiring preorder in which $a \le b$ if and only if $b - a \in M$. Then $S$ has quasi-complements due to the Archimedeanicity assumption $R = \N - R_+$.
\end{proof}

Although a bit convoluted, our arguments thus establish that the Krivine--Kadison--Dubois theorem can be derived in an elementary way from Strassen's Vergleichsstellensatz. Note that Strassen's original proof of the latter had made use of the former~\cite{strassen}. In any case, all of this underlines that our \Cref{simpler1} can be thought of as a non-Archimedean generalization of these results.

\newpage
\bibliographystyle{plain}
\bibliography{local_global_principle}

\end{document}